\documentclass{amsart}
\usepackage{latexsym,amsxtra,amscd,ifthen}
\usepackage{amsfonts}
\usepackage{verbatim}
\usepackage{amsmath}
\usepackage{amsthm}
\usepackage{amssymb}
\usepackage{xcolor}
\usepackage[normalem]{ulem}
\usepackage{cancel}

\numberwithin{equation}{section}

\mathchardef\mhyphen="2D

\theoremstyle{plain}
\newtheorem{theorem}{Theorem}[section]
\newtheorem*{theorem*}{Theorem}
\newtheorem{lemma}[theorem]{Lemma}
\newtheorem{proposition}[theorem]{Proposition}
\newtheorem{hypothesis}[theorem]{Hypothesis}
\newtheorem{corollary}[theorem]{Corollary}
\newtheorem{conjecture}[theorem]{Conjecture}
\newtheorem{notation}[theorem]{Notation}

\theoremstyle{definition}
\newtheorem{definition}[theorem]{Definition}
\newtheorem{example}[theorem]{Example}

\newtheorem{remark}[theorem]{Remark}
\newtheorem{question}[theorem]{Question}

\makeatletter              
\let\c@equation\c@theorem  
\makeatother

\newcommand{\bfl}{\mathfrak l}

\DeclareMathOperator{\ged}{ged}
\DeclareMathOperator{\gldim}{gldim}

\DeclareMathOperator{\Ext}{Ext} 
\DeclareMathOperator{\Tor}{Tor}

 \DeclareMathOperator{\ann}{ann}

\DeclareMathOperator{\injdim}{injdim}

\DeclareMathOperator{\GKdim}{GKdim}

\DeclareMathOperator{\CMreg}{CMreg} 
\DeclareMathOperator{\Extreg}{Extreg} 
\DeclareMathOperator{\Torreg}{Torreg}
\DeclareMathOperator{\reg}{reg}

\DeclareMathOperator{\Hom}{Hom}
\DeclareMathOperator{\RHom}{RHom}
\DeclareMathOperator{\GrMod}{{\sf Gr}}
 
\DeclareMathOperator{\im}{im}

\newcommand{\fm}{\mathfrak{m}}

\newcommand{\id}{\operatorname{id}}

\newcommand{\ch}{\operatorname{char}}

\newcommand{\be}{\begin{enumerate}}
\newcommand{\ee}{\end{enumerate}}
\newcommand{\bq}{\begin{eqnarray*}}
\newcommand{\eq}{\end{eqnarray*}}
\newcommand{\bqn}{\begin{eqnarray}}
\newcommand{\eqn}{\end{eqnarray}}
\newcommand{\op}{\text{op}}

\begin{document}
\title[Degree bounds for Hopf actions on AS regular algebras]
{Degree bounds for Hopf actions\\
on Artin--Schelter regular algebras}

\author{E. Kirkman, R. Won and J. J. Zhang}

\address{Kirkman: Department of Mathematics,
P. O. Box 7388, Wake Forest University, Winston-Salem, NC 27109}

\email{kirkman@wfu.edu}

\address{Won: Department of Mathematics, The George Washington University, Washington, DC 20052, USA}
\email{robertwon@gwu.edu}

\address{Zhang: Department of Mathematics, Box 354350,
University of Washington, Seattle, Washington 98195, USA}

\email{zhang@math.washington.edu}

\begin{abstract}
We study semisimple Hopf algebra actions on Artin--Schelter regular algebras
and prove several upper bounds on the degrees
of the minimal generators of the invariant subring, and on the degrees of syzygies of modules over the invariant subring.
These results are analogues of results for group actions on commutative polynomial rings proved by Noether, Fogarty, Fleischmann, Derksen, Sidman,  Chardin, and Symonds.
\end{abstract}
\makeatletter
\@namedef{subjclassname@2020}{%
  \textup{2020} Mathematics Subject Classification}
\makeatother


\maketitle

\setcounter{section}{-1}

\section{Introduction}
\label{xxsec0}

Throughout, let $\Bbbk$ be a field.
The invariant subring $T^G$ of a commutative polynomial 
ring $T:=\Bbbk[x_1, \dots, x_n]$ under the linear action 
of a group $G$ has played an important role in 
commutative algebra and algebraic geometry. Producing 
a minimal generating set for $T^G$ (as an algebra) is 
the first step in understanding the invariant subring. 
In 1916, Noether proved the 
following fundamental theorem, which is sometimes called 
Noether's upper bound theorem. 

\begin{theorem}[{\cite{Noe}}]
\label{xxthm0.1}
If $\Bbbk$ is a field of characteristic zero and $G$ is 
a finite group of invertible $n \times n$ matrices acting 
linearly on $T:= \Bbbk[x_1, \dots, x_n]$ then the ring of 
invariants $T^G$ can be generated as a $\Bbbk$-algebra by 
polynomials of total degree $\leq |G|$. 
\end{theorem}

This result is extremely useful in explicitly computing the invariant 
subring, for in characteristic zero, invariants are linear 
combinations of elements of the form $\sum_{g \in G} g.m$, 
 the sum of the elements in the $G$-orbit of a 
monomial $m \in T$. Knowing an upper bound on the degrees of 
minimal generators of the invariants affords an algorithm for 
finding the generators, as one can compute these sums for all 
monomials $m$ of degree less than or equal to the bound.

Noether's upper bound was extended to the non-modular case 
(where $\ch \Bbbk$ does not divide $|G|$) independently by 
Fleischmann \cite{Fl} in 2000, Fogarty \cite{Fo} in 2001, and 
Derksen and Sidman \cite{DS} in 2002. Derksen and Sidman used 
a homological invariant, the Castelnuovo--Mumford (CM) regularity of 
a subspace arrangement associated to the action, to obtain 
their bound. Hence in the non-modular case there is a bound on 
the degrees of the minimal generators that is independent of 
the representation of the group, though the actual degrees of 
the generators may be quite a bit less than the order of the 
group (e.g., Domokos and Heged\"{u}s \cite{DH} provided a smaller 
upper bound on the degrees of generators if $G$ is not a cyclic 
group, and this result was extended to the non-modular case by 
Sezer \cite{Se}). Surveys of results extending Noether's bound 
that were obtained before 2007 can be found in \cite{Neu} and 
\cite{We}.

In the modular case, 
the Noether bound does not hold; 
for example if $\ch \Bbbk = 2$, there is 
an action of the group of order 2 on $\Bbbk[x_1,x_2, x_3, x_4,x_5,x_6]$ that 
requires a generator of degree 3 \cite[Example 3.5.5(a)]{DK}.  
Using CM regularity (though 
differently than Derksen and Sidman \cite{DS}), Symonds proved a 
bound that depends on both the order of the group $G$ and the 
dimension of the representation of $G$: if $G$ acts on 
$T=\Bbbk[x_1, \dots, x_n]$ for $n >1$ and $|G| \geq 2$ then 
$T^G$ can be generated by elements of degree $\leq n (|G|-1)$
\cite[Corollary 0.2]{Sy}. Hence in the modular case the upper 
bound depends upon both $|G|$ and the degree of the representation,
or, equivalently, the global dimension of $T$.

It is also natural to look for bounds on the maximal degrees of 
syzygies of $T^G$ as a 
quotient module over another polynomial ring and the maximal 
degrees of syzygies of the 
trivial module $\Bbbk$ over $T^G$.
Castelnuovo--Mumford regularity 
again has been useful in finding such bounds \cite{De, CS, Sy}. 
A nice result of Chardin and Symonds \cite[Theorem 1.3]{CS}, 
generalizing work of Derksen \cite{De}, states that, if $\Bbbk$ 
has characteristic zero, then for all $i \geq 2$  $\beta_i(T^G)$, the maximal degree of 
$\Tor^{T^G}_i(\Bbbk,\Bbbk)$, is bounded by {$|G| i+i-2$}. 
There is also an approach using the theory of twisted commutative algebras,
which Snowden used to establish bounds on maximal degrees of syzygies \cite{Sn}.
Gandini used similar techniques
to extend the Noether bound (and syzygy bound) to a noncommutative 
ring, the exterior algebra, in the characteristic zero case 
\cite[Theorem V1.13]{Ga}.

Our main goal in this paper is to explore degree bounds on 
minimal generators, as well as of syzygies, for invariants in 
noncommutative algebras, where little is known about these bounds. 
We call an algebra $A$ {\em connected graded} 
if it has a $\Bbbk$-vector space decomposition
$$A = \Bbbk \oplus A_1  \oplus A_2  \oplus \cdots$$
with $1 \in A_0$, and $A_iA_j \subseteq  A_{i+j}$ for all 
$i, j \in \mathbb{N}$. Throughout let $A$ be a connected 
graded noetherian algebra. If $M$ is a graded 
(left or right) $A$-module, the shifted $A$-module $M(n)$ is the 
graded module defined by $M(n)_i = M_{n+i}$. Denote by 
$\beta(A) \in {\mathbb N}\cup \{\infty\}$ the largest degree of 
an element in a minimal generating set of $A$. 

In the noncommutative setting, it is natural to 
replace the commutative polynomial ring with a noetherian 
Artin--Schelter regular $\Bbbk$-algebra \cite{AS} generated in 
degree one, as such algebras share many of the homological 
properties of commutative polynomial rings, and, when commutative, 
these Artin--Schelter regular algebras are isomorphic to polynomial 
rings. 

\begin{definition}
\label{xxdef0.2}
A connected graded algebra $T$ is called {\it Artin--Schelter 
Gorenstein} (or {\it AS Gorenstein}, for short) if the following 
conditions hold:
\begin{enumerate}
\item[(a)]
$T$ has injective dimension $d<\infty$ on the left and on the 
right,
\item[(b)]
$\Ext^i_T({}_T\Bbbk, {}_{T}T)=\Ext^i_{T}(\Bbbk_T,T_T)=0$ for all
$i\neq d$, and
\item[(c)]
$\Ext^d_T({}_T\Bbbk, {}_{T}T)\cong \Ext^d_{T}(\Bbbk_T,T_T)\cong 
\Bbbk(\bfl)$ for some integer $\bfl$. Here $\bfl$ is called the 
{\it AS index} of $T$.
\end{enumerate}
In this case, we say $T$ is of type $(d,\bfl)$. If in addition,
\begin{enumerate}
\item[(d)]
$T$ has finite global dimension, and
\item[(e)]
$T$ has finite Gelfand--Kirillov dimension (see \eqref{E1.7.1}),
\end{enumerate}
then $T$ is called {\it Artin--Schelter regular} (or {\it AS
regular}, for short) of dimension $d$.
\end{definition}

Throughout, we will use the letters $S$ and $T$ to denote AS 
regular (or AS Gorenstein) algebras, while the letters $A$ and $B$ 
will usually be used for connected graded algebras, more generally. 

One reason we focus on AS Gorenstein algebras is that
noncommutative CM regularity [Definition~\ref{xxdef2.9}]
has been studied for these algebras by J{\o}rgensen 
\cite{Jo2, Jo3} and Dong and Wu \cite{DW}. 
As in the commutative case, we will see that CM regularity
is an important tool for proving bounds on the degrees of generators.
Conversely, results on degrees of generators 
contribute to the further understanding of 
noncommutative CM regularity. 

We will consider groups $G$ that act on AS regular algebras 
$T$ via graded automorphisms, and, more generally, semisimple 
Hopf algebras that act homogeneously on $T$. For a Hopf algebra 
$H$, we use the standard notation $\Delta: H \rightarrow  
H \otimes H$ for the coproduct, $\varepsilon: H \rightarrow 
\Bbbk$ for the counit, and $S: H \rightarrow H$ for the 
antipode of $H$. Further details on Hopf actions on algebras 
can be found in \cite{Mo}. In most cases, we will assume the 
following hypotheses.

\begin{hypothesis}
\label{xxhyp0.3}
\begin{enumerate}
\item[(a)] 
$H$ is a semisimple, hence finite-dimensional (by \cite{LZ}),
Hopf algebra,
\item[(b)] 
$T$ is a connected graded noetherian AS Gorenstein algebra of injective
dimension at least two, and
\item[(c)]
$T$ is a left $H$-module algebra and for each $i$, $T_i$ is 
a left H-submodule of $T$.
\end{enumerate}
\end{hypothesis}

\noindent Previous work shows that many results concerning group 
actions on commutative polynomial rings have generalizations 
to the context of Hopf actions on AS regular algebras 
(see, e.g., \cite{CKWZ1, CKWZ2, CKWZ3, CG, FKMW1, FKMW2, FKMP, Ki, KKZ1, KKZ2, 
KKZ3, KKZ6}).

Unfortunately, there is a serious lack of understanding of 
degree bounds in the noncommutative context. To illustrate 
this, note the following two facts that indicate that the
noncommutative case is quite different from the commutative case.
\begin{enumerate}
\item[(a)]
In the commutative case, if any finite group $G$ acts nontrivially
on the polynomial ring $\Bbbk[x_1,\cdots,x_n]$, then 
$\beta(\Bbbk[x_1,\cdots,x_n]^G) >1$.
However, when $T$ is noncommutative, $\beta(T^G)$ can be 1 even 
if $T$ is a Koszul AS regular algebra [Example \ref{xxex1.2}(2)].

\item[(b)]
In the commutative non-modular case, by Theorem \ref{xxthm0.1} 
and \cite{Fl, Fo, DS}, $\beta(\Bbbk[x_1,\cdots,x_n]^G) \leq |G|$. 
However, if $T$ is noncommutative, then $\beta(T^G)$ can be strictly 
larger than $|G|$.
In [Example \ref{xxex1.2}(3)], we provide an example of a $\mathbb{Z}/(2)$
action on $\Bbbk_{-1}[x_1,x_2]$ such that $\beta(\Bbbk_{-1}[x_1,x_2]^{\mathbb{Z}/(2)}) = 3$. 
Furthermore, Ferraro, Moore, Peng, and the first-named author showed
that for $n$ odd there is a cyclic group of order $2n$ acting  
on $T=\Bbbk_{-1}[x_1,x_2]$ with $\beta(T^G) = 3n$ \cite[Theorem 2.5]{FKMP}. 
Hence, the difference $\beta(T^G) - |G|$ can, in fact, be arbitrarily large.
\end{enumerate}

In our noncommutative context, an upper bound on the degrees 
of minimal algebra generators for the invariant subring of an 
AS regular algebra $T$ might depend on both the algebra $T$, as 
well as the group (or Hopf algebra) and its representation, while 
classical invariant theory is restricted to considering only the 
single commutative AS regular algebra $\Bbbk[x_1, \dots, x_n]$. 
In Example~\ref{xxex3.6}, we show that for any $m$, there is a 
2-generated noetherian AS regular algebra on which a group of 
order 2 acts such that the maximal degree of minimal generators 
of the invariant subring is at least $m$.  Hence, there is no 
bound on the maximal degree of a generating set of $T^G$ that 
holds for all noetherian AS regular algebras $T$ that is 
dependent upon only the number of generators of $T$ and the order 
of the group $G$. It would be nice to have a degree bound that 
does not depend on the action of $G$ (or more generally $H$) on 
$T$. We pose the following question.

\begin{question}
\label{xxque0.4}
Suppose $(T,H)$ satisfies Hypothesis \ref{xxhyp0.3}. Assume 
that $T$ is an AS regular domain. Is $\beta(T^H)$ --- the 
maximal degree of a minimal generating set of $T^H$ --- 
bounded by a function of the numerical invariants
$$\dim_{\Bbbk} H, \quad \gldim T, \quad \CMreg(T)?$$
\end{question}

We are able to answer this
question in some special cases.
First, we generalize a commutative result
and show that $\beta(T^H)$
is bounded by $\tau_H(T)$, the $\tau$-saturation degree 
of the $H$-action on $T$ (see Definition~\ref{xxdef1.1}(2)).

\begin{theorem}[{Corollary~\ref{xxcor3.3}}]
\label{xxthm0.5b}
Let $A$ be a connected graded algebra and let $H$ be a 
semisimple Hopf algebra acting on $A$ homogeneously. Then 
$\beta(A^H) \leq \tau_H(A)$. 
\end{theorem}

This result is computationally useful, as
 $\tau_H(T)$ is often easy to bound.
We are able to use this result to answer Question~\ref{xxque0.4} 
in the following two cases. Note that $\beta_i$ will be 
defined in \eqref{E1.0.6}.

\begin{theorem}[Theorem \ref{xxthm3.5}]
\label{xxthm0.5}
Suppose that $(T,H)$ satisfies Hypothesis \ref{xxhyp0.3}. 
Assume further that
\begin{enumerate}
\item[(a)]
$T$ is an AS regular domain generated 
in degree 1 such that $T\# H$ is prime, and
\item[(b)]
$T^H$ has finite global dimension.
\end{enumerate}
Then $\beta(T^H)\leq \dim H$.
\end{theorem}

\begin{theorem}
\label{xxthm0.6}
Suppose $\Bbbk$ is an infinite field, $G$ is 
a finite group, and $\Bbbk G$ is semisimple.  
Suppose $G$ acts via graded automorphisms 
on $\Bbbk_{-1}[x_1, \dots, x_n]$. 
\begin{enumerate}
\item[(1)] {\rm{[Corollary \ref{xxcor3.12}]}} Then
$$\beta(\Bbbk_{-1}[x_1, \dots, x_n]^G) \leq 2|G| +n.$$
\item[(2)] {\rm{[Corollary \ref{xxcor5.10}]}}
Assume 
that $\Bbbk_{-1}[x_1, \dots, x_n]^G$ is commutative.  Then
$$\beta_i(\Bbbk_{-1}[x_1, \dots, x_n]^G)\leq i (2|G|+n+1)-2$$
for all $i\geq 2$.
\end{enumerate}
\end{theorem}

An interesting open question is whether there is a similar 
bound for Hopf algebra actions on $\Bbbk_{q}[x_1, \dots, x_n]$ 
where $q$ is a root of unity (the only interesting group 
actions occur when $q=\pm 1$ while there are Hopf actions 
for other values of $q$).

One difference between the commutative and noncommutative 
cases is that when a group $G$ acts on $C:= \Bbbk[x_1, \dots, x_n]$ 
via graded automorphisms, there is a graded surjective map 
from a polynomial ring onto $C^G$. On the other hand, if 
a Hopf algebra $H$ acts on an AS regular algebra $T$ as in 
Hypothesis~\ref{xxhyp0.3}, it is not known if there always 
exists an analogous AS regular algebra $S$ mapping onto $T^H$
(see Remark \ref{xxrem5.7}).  

Moreover, in the commutative case, by the Noether Normalization 
Theorem, there exists a polynomial subring $B$ of $C^G$ such 
that $C^G$ is a finitely generated $B$-module. In some cases, 
it is known that an analogous AS regular subalgebra exists 
(e.g., see Lemma \ref{xxlem3.8}). The following result provides 
bounds on the degrees of a minimal generating set of the 
invariant subring and on the relations among the generators 
in the case when there is a graded algebra map from an AS regular algebra 
$S$ to $T^H$; it is a noncommutative version of 
\cite[Proposition 2.1(2, 3)]{Sy}.

\begin{theorem}[Corollary \ref{xxcor4.6}]
\label{xxthm0.7}
Let $(T,H)$ be as in Hypothesis \ref{xxhyp0.3}. Suppose there is a graded 
algebra map $S\to T^H$ where $S$ is a noetherian AS regular 
algebra such that $T^H$ is finitely generated over $S$ on
both sides. Let $\delta(T/S)=\CMreg(T)-\CMreg(S)$. Then
\begin{enumerate}
\item[(1)]
$$\beta(T^H)\leq \max\{\beta(S), \delta(T/S)\} \text{ and}$$ 
\item[(2)]
$$
\beta_2(T^H)\leq \max\left\{2 \delta(T/S),
\delta(T/S)+\beta(S),\beta_2(S)\right\}.$$
\end{enumerate}
\end{theorem}

We also provide a noncommutative version of 
\cite[Theorem 1.2 (1)]{CS}, which gives bounds on the maximal 
degrees in a projective $S$-module resolution of $T^H$ when 
there is a graded surjection from an AS regular algebra $S$ 
onto $T^H$. Let $J_i$ denote the annihilator ideal of the 
finite-dimensional left $T$-module $\Tor^{S}_i(T, \Bbbk)$ and 
let $J_{\infty}=\cap_{j\geq 0} J_j$. Note that $t^S_i$ will be 
defined in \eqref{E1.0.3}--\eqref{E1.0.4}.

\begin{theorem}[Theorem \ref{xxthm5.6}]
\label{xxthm0.8}
Let $(T,H)$ be as in Hypothesis \ref{xxhyp0.3} and assume 
that $T$ is Koszul. Suppose there exists a graded algebra 
surjection $S \to T^H=:R$ where $S$ is a noetherian AS regular 
algebra such that $t^S_j(\Bbbk)\leq  (\deg T/J_{\infty}+2)j$ 
for all $j\geq 0$. Then 
$$t^S_j(R_S)\leq (\deg T/J_{\infty}+2)j+\deg T/J_{\infty}$$
for all $j\geq 0$.
\end{theorem}

In practice, one may bound $\deg T/J_{\infty}$ by other means: 
see, for example, Proposition \ref{xxpro3.11}(2).
The following is a noncommutative version of \cite[Theorem 2]{De}.

\begin{theorem}[Theorem \ref{xxthm5.11}]
\label{xxthm0.9}
Let $(T,H)$ be as in Hypothesis \ref{xxhyp0.3} and suppose that
$T$ is AS regular. Suppose further that
\begin{enumerate}
\item[(a)]
$T$ is generated in degree 1 and
\item[(b)]
$S$ is a noetherian AS regular algebra such that the minimal generating 
vector spaces of $S$ and $R:=T^H$ have the same dimension
and there exists a graded algebra surjection $S \to R$.
\end{enumerate}
Then the following statements hold.
\begin{enumerate}
\item[(1)]
We have 
$$\beta_2(R):=t^R_2(\Bbbk)\leq 2-2\CMreg(S)+\CMreg(T).$$
\item[(2)]
Suppose that $\Tor^S_1(\Bbbk, R)\otimes_R \Bbbk \cong 
\Tor^S_1(\Bbbk, R)$. 
Then 
$$t^S_1(_S R)\leq 2-2\CMreg(S)+\CMreg(T).$$
\item[(3)]
Suppose the hypothesis of part {\rm{(2)}}.
Let $K$ be the kernel of the algebra map $S\to R$. Then, as a left
ideal of $A$, $K$ is generated in degree at most
$$2-2\CMreg(S)+\CMreg(T).$$
\end{enumerate}
\end{theorem}

Both $t_1^S({}_S R)$ and $t_2^R(\Bbbk)$ give information on 
the degrees of the relations between the minimal generators of $R$
(when viewing $R$ as an $S$-module or as a $\Bbbk$-algebra),
and this theorem provides a bound for both $t_1^S({}_S R)$ and $t_2^R(\Bbbk)$
which depends only on the CM regularities of $S$ and $T$.
Other degree bounds for higher syzygies can be found in Corollary 
\ref{xxcor4.8} (bounds on $t_i^S(_SR)$) and Proposition 
\ref{xxpro5.8} (bounds on $t_i^C(\Bbbk)$ where $C= f(A)$ is the 
image of a connected graded algebra).
We remark that some of the bounds that we obtain
for actions on noncommutative AS regular algebras 
differ from the analogous bounds in the classical case 
(e.g., Noether's bound does not hold), but many of our 
bounds reduce to the results known for group actions 
on commutative polynomial rings.

The paper is organized as follows. Section 1 contains basic 
definitions and key examples. In Section~\ref{xxsec2}, we review 
the properties of (noncommutative) local cohomology that are 
needed in the paper and present J{\o}rgensen's definition of 
noncommutative Castelnuovo--Mumford regularity.
In Section~\ref{xxsec3}, we obtain some bounds on $\beta(A^H)$ 
using a noncommutative version of $\tau$-saturation degree that 
was introduced in \cite{Fo}. In Sections~\ref{xxsec4} and 
~\ref{xxsec5}, assuming the existence of a graded algebra map 
from an AS regular algebra $S$ to $T^H$, we prove bounds on the 
degrees of the projective modules in free resolutions obtaining, 
in Section~\ref{xxsec4}, generalizations of results of Derksen and 
Symonds, and, in Section~\ref{xxsec5}, results of Chardin and 
Symonds and of Derksen that can be used to provide bounds on the 
higher syzygies of $T^H$ as an $S$-module.  
We conclude in 
Section~\ref{xxsec6} by listing questions for further study.

\section{Hilbert ideals and $\tau$-saturation degree}
\label{xxsec1}

Let $M=\bigoplus_{d\in {\mathbb Z}} M_d$ be a ${\mathbb Z}$-graded 
$\Bbbk$-vector space. Define the {\it degree} of $M$ to be the 
maximum degree of a nonzero homogeneous element in $M$, namely, 
\begin{equation}
\label{E1.0.1}\tag{E1.0.1}
\deg(M)=\inf\{d\mid  (M)_{\geq d}= 0\}-1=
\sup\{d\mid  (M)_{d}\neq 0\} \quad \in \quad  
{\mathbb Z} \cup\{\pm \infty\}.
\end{equation}
Similarly, we define
\begin{equation}
\label{E1.0.2}\tag{E1.0.2}
\ged(M)=\sup\{d\mid  (M)_{\leq d}= 0\}+1=
\inf\{d\mid  (M)_{d}\neq 0\} \quad \in \quad  
{\mathbb Z} \cup\{\pm \infty\}.
\end{equation}
We say that a graded module $M$ is {\em locally-finite} 
if $\dim_{\Bbbk} M_d<\infty$ for all $d\in {\mathbb Z}$. 
Let $A$ be a connected graded algebra with trivial $A$-bimodule 
denoted by $\Bbbk$. For a graded left $A$-module $M$, let
\begin{equation}
\label{E1.0.3}\tag{E1.0.3}
t^A_i(_A M)=\deg \Tor^A_i(\Bbbk, M).
\end{equation}
If $M$ is a right graded $A$-module, let
\begin{equation}
\label{E1.0.4}\tag{E1.0.4}
t^A_i(M_A)=\deg \Tor^A_i(M,\Bbbk).
\end{equation}
It is clear that $t^A_i(_A \Bbbk)=t^A_i(\Bbbk_A)$.
If the context is clear, we will use $t^A_i(M)$
instead of $t^A_i(_A M)$ (or $t^A_i(M_A)$).

Recall that $\beta(A)$ denotes the largest degree of elements 
in a minimal generating set of $A$. This number is independent 
of the choice of the minimal generating set since it is equal 
to the largest degree of elements in the graded vector space 
$\Tor^A_1(\Bbbk, \Bbbk)$, namely,
\begin{equation}
\label{E1.0.5}\tag{E1.0.5}
\beta(A)=t^A_1(\Bbbk)=\deg (\Tor^A_1(\Bbbk, \Bbbk)).
\end{equation}
More generally, for any $i \geq 2$, we define 
\begin{equation}
\label{E1.0.6}\tag{E1.0.6}
\beta_i(A)=t^A_i(\Bbbk)=\deg (\Tor^A_i(\Bbbk,\Bbbk)).
\end{equation}
While $\beta(A)$ gives information about degrees of 
generators of $A$, $\beta_2(A)$ yields information 
about the degrees of the relations of $A$.

Let $H$ be a semisimple Hopf algebra acting on $A$ 
homogeneously. Sometimes we will use $R$ to denote 
the {\it invariant subring}  (or {\it fixed subring})
\begin{equation}
\label{E1.0.7}\tag{E1.0.7}
R:= A^H:=\{a\in A\mid h\cdot a =\varepsilon(h) a\}.
\end{equation}
Since the action of $H$ on $A$ is homogeneous, 
$R$ is a connected graded algebra. We note that 
$\beta(R):=\beta(A^H)$ generally depends on $A$, $H$ 
and the action of $H$ on $A$.

We now define the Hilbert ideal and the $\tau$-saturation 
degree, which were used in the non-modular proofs of the 
Noether bound by Fogarty \cite{Fo} and Fleischmann \cite{Fl}. 
These definitions extend easily to noncommutative algebras. 
In the noncommutative case, the left Hilbert ideal was 
introduced in Gandini's thesis \cite[Definition II.11]{Ga}.

\begin{definition}
\label{xxdef1.1} Retain the above notation.
\begin{enumerate}
\item[(1)]
The {\it left Hilbert ideal} of the $H$-action on $A$ is 
the left ideal of $A$
$$J_H(A):=A R_{\geq 1}.$$
The {\it right Hilbert ideal} of the $H$-action on $A$ is 
the right ideal of $A$
$$J^{\op}_H(A):=R_{\geq 1} A.$$
\item[(2)]
The {\it $\tau$-saturation degree} (or {\it left $\tau$-saturation 
degree}) of the $H$-action on $A$ is defined to be
$$\tau_H(A)=1+\deg(A/J_H(A))=1+t^R_0(A_R) 
\quad \in \quad {\mathbb N} \cup\{\infty\}.$$
The {\it right $\tau$-saturation degree} of the $H$-action 
on $A$ is defined to be
$$\tau^{\op}_H(A)=1+\deg(A/J^{\op}_H(A))=1+t^R_0(_RA) 
\quad \in \quad {\mathbb N} \cup\{\infty\}.$$
\item[(3)]
For every integer $i \geq 0$, the {\it $i$th annihilator ideal} 
of the $H$-action on $A$ is defined to be the (two-sided) ideal 
of $A$
$$J_{H,i}(A):=\ann (_A \Tor^R_i(A_R,\Bbbk))$$
and let
$$J_{\infty}=\bigcap_{i\geq 0} J_{H,i}.$$
\end{enumerate}
\end{definition}

It is easy to see that $J_{H,0}(A)\subseteq J_H(A)$ and 
$\deg A/J_{H,0}(A)=\deg A/J_H(A)$. It is not known if 
$\tau_H(A)=\tau^{\op}_H(A)$ in general (see Question 
\ref{xxque6.1}). 

\begin{example}
\label{xxex1.2}
\begin{enumerate}
\item[(1)]
Suppose that $\Bbbk$ contains a primitive $m$th root of 
unity $\omega$ and that $A(\neq \Bbbk)$ is generated in 
degree 1. Define the map $\sigma(a) = \omega a$ for all 
$a \in A_1$ and extend it to an automorphism of $A$. Then the 
invariant subring $A^{\langle \sigma \rangle}$ is the 
$m$th Veronese ring of $A$ so the left Hilbert ideal 
$J_{\langle \sigma \rangle}(A)$ is zero in degrees $< m$ 
and equal to $A$ in degrees $\geq m$. As a result, 
$\tau_{\langle \sigma \rangle}(A) = m$. If $A$ is a domain,
then $\beta (A^{\langle \sigma \rangle})=m$. Similarly, 
$\tau^{\op}_{\langle \sigma \rangle}(A) = m$.
\item[(2)]
\cite[Example 5.4(c)]{KKZ1}
Let $T$ be the Rees ring of the first Weyl algebra with 
respect to the standard filtration. So $T$ is generated 
by $x$, $y$, and $z$ subject to the relations
$$xy-yx = z^2, \quad {\text{$z$ is central.}}$$
Let $\sigma$ be the automorphism of $T$ determined by
$$\sigma(x) = x, \quad \sigma(y) = y, \quad {\text{and}} \quad 
\sigma(z) = -z.$$
By \cite[Example 5.4(c)]{KKZ1}, $T^{\langle \sigma \rangle}$
is generated in degree 1. In this case 
$$\beta(T^{\langle \sigma \rangle})=1<2 =
\tau_{\langle \sigma \rangle}(T).$$
Note that \cite[Conjecture 0.3]{FKMW1} fails as 
$T^{\langle \sigma\rangle}$ is AS regular and the product of 
the degrees of a homogeneous minimal generating set is 1 and 
$\dim \Bbbk \langle \sigma\rangle=2$. 
\item[(3)]
\cite[Example 3.1]{KKZ6}
Let $T$ be the $(-1)$-skew polynomial ring $\Bbbk_{-1}[x_1,
x_2]$ and $\sigma$ be the automorphism of $T$ exchanging 
$x_1$ and $x_2$. Let $G=\langle \sigma\rangle$. By 
\cite[Example 3.1]{KKZ6}, $T^G$ is generated by $x_1+x_2$
and $x_1^3+x_2^3$, so $\beta(T^G)=3>2=|G|$. One can easily 
check that $\tau_{G}(T)=\beta(T^G)$ in this case.
\end{enumerate}
\end{example}

We remark that $\tau_H(A)$ (and $\tau^{\op}_H(A)$) need 
not be finite when $A$ is not noetherian, as the next 
example shows.

\begin{example}
\label{xxex1.3} 
Let $A$ be the free algebra generated by two elements, 
say $x$ and $y$, in degree 1. Let $G$ be the group
${\mathbb Z}/(2)=\{ \id, \sigma \}$ and define an action 
of $G$ on $A$ by $$ \sigma(x)=-x, \quad \sigma(y)=y.$$
It is clear that the invariant subring $A^G$ is generated
by a minimal generating set of homogeneous elements
$$\{y, x^2, xyx, xy^2x, xy^3x, \cdots, xy^n x, \cdots\}.$$ 
In particular, $\beta(A^G)=\infty$. By a general result 
in Section \ref{xxsec3} (Corollary \ref{xxcor3.3}), 
one sees that $\tau_G(A)\geq \beta(A^G)=\infty$, but
we can check this  directly. Note that $A$ is an infinitely
generated right $A^G$-module with a minimal generating set
$$\Phi:=\{1, x, yx, y^2x, y^3x, y^4x, \cdots\}.$$
To see this, observe that every homogeneous element 
can be written as $y^n f$ where $f$ contains an even 
number of $x$'s or $y^n x f$ where $f$ contains an even 
number of $x$'s. Thus $\Phi$ generates $A$ as a right 
$A^G$-module. Further, every $y^n x$ cannot be written 
as an element in $A^G+\sum_{i=0}^{n-1} y^i x A^G$. 
Therefore $\Phi$ is a minimal generating set of 
the right $A^G$-module $A$. This implies that 
$A\otimes_{A^G} \Bbbk$ is infinite-dimensional, 
whence, $\tau_G(A)=\infty$. Similarly, 
$\tau^{\op}_G(A)=\infty$.
\end{example}

We are mainly interested in noetherian algebras, and 
in that case the $\tau$-saturation degree is always 
finite.

\begin{lemma}
\label{xxlem1.4} 
Retain the above notation. Suppose that $A$ is noetherian.
\begin{enumerate}
\item[(1)]
Both $t^R_i(A_R)$ and $t^R_i(_RA)$ are finite for all 
$i\geq 0$. 
\item[(2)]
Both $\tau_H(A)$ and $\tau^{\op}_H(A)$ are finite.
\end{enumerate}
\end{lemma}

\begin{proof}
Part (2) is a special case part (1), so we only need to prove 
part (1). By symmetry, it suffices to show that 
$t^R_i(A_R)$ is finite for all $i\geq 0$. 

By \cite[Corollary 4.3.5 and Theorem 4.4.2]{Mo}, $R:=A^H$ 
is noetherian, and $A$ is a finitely generated module over 
$R$ on both sides. Since $R$ is noetherian, every term in 
the minimal free resolution $P^{\bullet}$ of $A_R$ is a 
finitely generated free right $R$-module. Thus, each term 
in $P^{\bullet} \otimes_R \Bbbk$ is finite-dimensional.
Consequently, $\Tor^R_i(A,\Bbbk)$ is finite-dimensional and 
the assertion follows.
\end{proof}

\begin{proposition}
\label{xxpro1.5}
Let $A$ be a noetherian domain and let $H$ be a semisimple 
Hopf algebra acting on $A$ homogeneously. Suppose that $A\# H$ 
is prime. Then $A$ is a finitely generated left and right
$A^H$-module of rank equal to $\dim H$.

If, in addition, $A$ and $A^H$ are AS regular, then $A$ is a 
free $A^H$-module.
\end{proposition}

\begin{proof} 
By \cite[Corollary 4.3.5 and Theorem 4.4.2]{Mo}, $R:=A^H$ 
is noetherian, and $A$ is a finitely generated module 
over $R$ on both sides.

Let $Q$ be the quotient division ring of $A$. Since $H$ 
acts on $A$, this action extends naturally to an $H$ action on $Q$.
Further, $Q\# H$ is the artinian quotient ring of the prime ring $A \# H$
and so $Q\# H$ is simple. By \cite[Corollary 3.10]{CFM}, 
$$[Q:Q^H]_{r}=[Q:Q^H]_{l}=\dim H.$$ 
By \cite[Theorem 4.3(iii)]{Sk}, $Q^H=Q(A^H)$. Since 
the (left) rank of $A$ over $A^H$ is equal to the 
(left) rank of $Q$ over $Q^H$, the assertion follows. 

If $A$ and $A^H$ are AS regular, then by 
\cite[Lemma 1.10(a, b)]{KKZ1}, $A$ is a finitely 
generated free module over $R$ on both sides. 
\end{proof}

See \cite[Lemma 3.10]{BHZ} for a condition when 
$A\# H$ is prime.

In Corollary~\ref{xxcor3.3} in Section \ref{xxsec3}, we will 
see that the $\tau$-saturation degree of the $H$-action 
on $A$ gives a bound on $\beta(A^H)$, so it is 
important to bound the $\tau$-saturation degree. Next is an 
easy example.

\begin{example}
\label{xxex1.6}
Let $T$ be a noetherian AS regular domain generated 
in degree 1. Suppose that $T^H$ has finite global
dimension and that $T\# H$ is prime. By \cite[Lemmas 1.10(c) and 1.11(b)]{KKZ1}, $T^H$ is AS regular. Hence,
by Proposition \ref{xxpro1.5}, $T$ is 
a free module over $R:=T^H$ of rank equal to $d:=\dim H$.
Then $M:=T\otimes_{R} R/R_{\geq 1}$ has $\Bbbk$-dimension 
$d$. As a left $T$-module, $M$ is generated by the element 
$1$ of degree 0. Since $T$ is generated in degree 1, 
$M_i\neq 0$ for all $i$ between 0 and 
$\sup\{j\mid  (T/J_H(T))_{j}\neq 0\}$.
So $\sup\{j\mid  (T/J_H(T))_{j}\neq 0 \}\leq d-1$ 
as $\dim M=d$. As a consequence, $\tau_H(T)\leq d$.
In summary, 
$$\tau_H(T)\leq \dim H.$$
Similarly, we have $\tau^{\op}_H(T)\leq \dim H$.
\end{example}

The Shephard--Todd--Chevalley Theorem describes the 
invariant subring of a commutative polynomial ring under 
the action of a reflection group; it shows that the 
invariant subring has finite global dimension, and the 
product of the degrees of the minimal generating invariants 
is equal to the order of the group. A similar phenomenon was observed 
in examples obtained in the noncommutative setting for group 
actions \cite{KKZ2} and Hopf actions \cite{FKMW1,FKMW2} where 
the invariant subrings have finite global dimension. We 
conclude this section with Proposition \ref{xxpro1.8}; 
parts (3) and (4) of  this proposition provides a generalization of 
Proposition \ref{xxpro1.5}, and describe some conditions
under which the product of the degrees of the minimal generators 
of the invariant subring is equal to the dimension of the 
Hopf algebra. It provides a partial answer to a variation of
\cite[Conjecture 0.3]{FKMW1}.

\begin{definition}
\label{xxdef1.7}
Let $A$ be a locally-finite, connected graded algebra 
$A:=\bigoplus_{i\geq 0} A_i$.
The {\it Hilbert series of $A$} is defined to be
\begin{equation*}
h_A(t)=\sum_{i\in {\mathbb N}} (\dim_{\Bbbk} A_i)t^i.
\end{equation*}
Similarly, if $M$ is a $\mathbb{Z}$-graded $A$-module 
(or $\mathbb{Z}$-graded vector space) 
$M = \bigoplus_{i \in \mathbb{Z}} M_i$, the 
\emph{Hilbert series of $M$} is defined to be
\begin{equation*}
h_M(t)=\sum_{i\in \mathbb{Z}} (\dim_{\Bbbk} M_i)t^i.
\end{equation*}
\end{definition}
The {\it Gelfand--Kirillov dimension} (or {\it GK dimension}) 
of a connected graded, locally-finite
algebra $A$ is defined to be
\begin{equation}
\label{E1.7.1}\tag{E1.7.1}
\GKdim (A)=\limsup_{n\to\infty} 
\frac{\log (\sum_{i=0}^{n} \dim_{\Bbbk} A_i)}{\log(n)},
\end{equation}
see \cite[Chapter 8]{MR}, \cite{KL}, or \cite[p.1594]{StZ}.

\begin{proposition}
\label{xxpro1.8}
Let $A$ be a noetherian connected graded domain and $H$ 
be a semisimple Hopf algebra acting 
on $A$ homogeneously. Assume that $A\# H$ is prime. 
\begin{enumerate}
\item[(1)]
Suppose the Hilbert series $p(t):=h_A(t)$ and $q(t):=
h_{A^H}(t)$ are rational functions. Then 
$$\big( p(t)/q(t)\big)\mid_{t=1}=\dim H.$$
\item[(2)]
If $\displaystyle h_A(t)=\frac{1}{\prod_{i=1}^n (1-t^{a_i})}$ 
and 
$\displaystyle h_{A^H}(t)=\frac{1}{\prod_{i=1}^n (1-t^{b_i})}$, 
then 
$$\prod_{i=1}^n b_i=(\dim H) \prod_{i=1}^{n} a_i.$$
\item[(3)]
If $\displaystyle h_A(t)=\frac{1}{(1-t)^n}$,
$\displaystyle h_{A^H}(t)=
\frac{1}{\prod_{i=1}^n (1-t^{b_i})}$, and 
there is a  minimal generating set 
$\{x_i\}_{i=1}^n$ of $A^H$ with $\deg x_i=b_i$, then 
$$\prod_{i=1}^n \deg x_i=\dim H.$$
\item[(4)]
Suppose that $\displaystyle h_A(t)=\frac{1}{(1-t)^n}$
and that $A^H$ is a commutative polynomial ring 
generated by a minimal homogeneous generating set 
$\{x_i\}_{i=1}^n$, then 
$$\prod_{i=1}^n \deg x_i=\dim H.$$
\end{enumerate}
\end{proposition}

\begin{proof}
(1) By Proposition~\ref{xxpro1.5}, $A$ is an $A^H$-module of 
rank $d := \dim H$ on both sides. Let $n$ be the GK dimension 
of $A$ (and equal to the GK dimension of $A^H$). Then, by 
\cite[Corollary 2.2]{StZ}, 
$$p(t)=p_1(t) (1-t)^{-n}, \quad {\text{and}} 
\quad 
q(t)=q_1(t) (1-t)^{-n}$$
such that $p_1(1)q_1(1)\neq 0$. 

Let $D(A)$ denoted the graded total quotient ring
of $A$. Then $D(A)$ is free over $D(A^H)$ of rank 
$d$ as $D(A^H)$ is a graded division ring. 
Then there is a right graded free $A^H$-submodule of
$A$ of rank $d$, say $M$,  such that 
$A/M$ is a torsion graded right $A^H$-module. Then  
$\GKdim A/M< \GKdim A^H=n$, and 
by \cite[Corollary 2.2]{StZ}, 
$$h_{A/M}(t)=h(t) (1-t)^{-m}$$
with $m<n$ and $h(1)\neq 0$. This implies that
$$\big( (1-t)^{n} h_{A/M}(t)\big) \mid_{t=1}=0.$$
Since $h_{A/M}(t)=h_A(t)-h_{M}(t)$, we have 
\begin{equation}
\label{E1.8.1}\tag{E1.8.1}
\big( (1-t)^{n} h_{A}(t)\big) \mid_{t=1}
=\big( (1-t)^{n} h_{M}(t)\big) \mid_{t=1}.
\end{equation}
The left-hand side of \eqref{E1.8.1} is $p_1(1)$. 
Since $M$ is free of rank $d$, there are integers 
$n_1,\cdots, n_d$ such that 
$$h_{M}(t)=(t^{n_1}+\cdots +t^{n_d}) h_{A^H}(t).$$
Then the right-hand side of \eqref{E1.8.1} is
$d q_1(1)$. As a consequence, $p_1(1)=dq_1(1)$. Now
$$\big( p(t)/q(t) \big)\mid_{t=1} 
=\big( p_1(t)/q_1(t) \big)\mid_{t=1} 
=p_1(1)/q_1(1)=d=\dim H.$$

Part (2) is clearly a consequence of part (1). 
Part (3) is clearly a consequence of part (2). 
Part (4) is clearly a consequence of part (3).
\end{proof}

We remark that the preceding result is not
true if we remove the hypothesis that $H$ is semisimple \cite[Observation 4.1(4)]{CWZ}.

\section{Preliminaries on local cohomology and CM regularity}
\label{xxsec2}

\subsection{Local cohomology}
\label{xxsec2.1}
In this subsection we will review some basic ideas related to graded 
modules, local cohomology, balanced dualizing complexes, and other concepts that will be used in discussing 
Castelnuovo--Mumford regularity. 
The definition of the Hilbert series $h_M(t)$ of a module $M$ was 
given in Definition~\ref{xxdef1.7}. If $M$ is locally-finite 
bounded above, then $h_M(t)$ is in $\Bbbk((t^{-1}))$. If $M$ is 
locally-finite bounded below, then $h_M(t)$ is in $\Bbbk((t))$. 
In both cases, some special $h_M(t)$ can be written as rational 
functions of $t$. 

\begin{definition}[{\cite[Definition 3.6.13]{BH}}]
\label{xxdef2.1} 
Let $M$ be nonzero locally-finite and either bounded above or 
bounded below. Suppose $h_M(t)$ is equal to a rational function, 
considered as an element in $\Bbbk ((t^{-1}))$ or in 
$\Bbbk ((t))$. The {\it $a$-invariant} of $M$, denoted by 
$a(M)$, is defined to be the $t$-degree of the rational 
function $h_M(t)$, namely, $a(M)=\deg_t h_{M}(t).$
\end{definition}

\begin{example}
\label{xxex2.2} 
\begin{enumerate}
\item[(1)]
If $M$ is finite-dimensional, then $a(M)=\deg(M)$. 
A more general case is considered in part (3).
\item[(2)]
Let $A$ be the commutative polynomial ring 
$\Bbbk[t_1,\cdots,t_n]$ with $\deg (t_i)=1$ for all
$i$. Then $h_{A}(t)=\frac{1}{(1-t)^{n}}$. 
Therefore, $a(A)=-n$.
\item[(3)]
If $M$ is bounded above and $h_M(t)$ is a rational function,
then one can check that
\begin{equation}
\label{E2.2.1}\tag{E2.2.1}
a(M)=\deg_t h_M(t)=\deg M.
\end{equation}
\end{enumerate}
\end{example}

Local cohomology is an important tool in this paper. Let 
$A$ be a locally-finite ${\mathbb N}$-graded algebra and 
let $\fm$ denote the graded ideal $A_{\geq 1}$. Let 
$A\mhyphen\GrMod$ denote the category of ${\mathbb Z}$-graded 
left $A$-modules. For each graded left $A$-module $M$, we define
$$\Gamma_{\fm}(M) 
=\{ x\in M\mid A_{\geq n} x=0 \; {\text{for some $n\geq 1$}}\;\}
=\lim_{n\to \infty} \Hom_A(A/A_{\geq n}, M)$$
and call this the {\it $\fm$-torsion submodule} of $M$. It 
is standard that the functor $\Gamma_{\fm}(-)$ is a left 
exact functor from $A\mhyphen\GrMod \to A\mhyphen\GrMod$. 
Since this category has enough injectives, the $i$th right 
derived functors, denoted by $H^i_{\fm}$ or $R^i\Gamma_{\fm}$, 
are defined and called the {\it local cohomology functors}. 
Explicitly, one has 
$$H^i_{\fm}(M)=R^i\Gamma_{\fm}(M)
:=\lim_{n\to \infty} \Ext^i_A(A/A_{\geq n}, M).$$ 
See \cite{AZ, VdB} for more details.

If $M$ is a left (or right) 
$A$-module, let $M'$ denote the graded $\Bbbk$-linear dual of 
$M$, where
\[ (M')_i = \Hom_{\Bbbk} (M_{-i}, \Bbbk).
\]

\begin{definition}
\label{xxdef2.3} 
Let $A$ be a locally-finite noetherian ${\mathbb N}$-graded
algebra. Let $M$ be a finitely generated graded left 
$A$-module. We call $M$ {\it $s$-Cohen--Macaulay} or simply 
{\it Cohen--Macaulay} if $H^i_{\fm}(M) = 0$ for all $i \neq s$ 
and $H^s_{\fm}(M) \neq 0$.
\end{definition}

The noncommutative version of a dualizing complex was introduced 
in 1992 by Yekutieli \cite{Ye1}. Roughly speaking, a dualizing 
complex over a connected graded algebra $A$ is a complex $R$ of 
graded $A$-bimodules, such that the two derived functors 
$\RHom_A(-,R)$ and $\RHom_{A^{\op}} (-,R)$ induce a duality between 
derived categories ${\mathbb{D}}^b_{f.g.}(A\mhyphen\GrMod)$ and 
${\mathbb{D}}^b_{f.g.}(A^{\op}\mhyphen\GrMod)$. Let $A^e$ denote 
the enveloping algebra $A\otimes A^{\op}$.

\begin{definition}[{\cite[Definition 3.3]{Ye1}}]
\label{xxdef2.4} 
Let $A$ be a noetherian connected graded algebra. A complex 
$R\in {\mathbb D}^b(A^{e}\mhyphen\GrMod)$ is called a {\it dualizing 
complex} over $A$ if it satisfies the three conditions below:
\begin{enumerate}
\item[(i)]
$R$ has finite graded injective dimension on both sides.
\item[(ii)]
$R$ has finitely generated cohomology modules on both sides.
\item[(iii)]
The canonical morphisms $A\to \RHom_{A}(R,R)$ and 
$A\to \RHom_{A^{\op}}(R,R)$ in ${\mathbb D}(A^{e}\mhyphen\GrMod)$
are both isomorphisms.
\end{enumerate}
\end{definition}

Note that  local cohomology  can be 
defined for complexes (see e.g. \cite{Jo1}). 
The following concept will be important in what follows. 

\begin{definition}[{\cite[Definition 4.1]{Ye1}}]
\label{xxdef2.5} 
Let $A$ be a noetherian connected graded algebra and 
let $R$ be a dualizing complex over $A$. We say 
that $R$ is a {\it balanced dualizing complex} if 
$$R\Gamma_{\fm} (R) \cong A'$$
in ${\mathbb D}(A^{e}\mhyphen\GrMod)$. 
\end{definition}

We recall a result of \cite{KKZ3}.

\begin{lemma}[{\cite[Lemma 3.2(b)]{KKZ3}}]
\label{xxlem2.6}
Assume Hypothesis \ref{xxhyp0.3}. Then
\begin{enumerate}
\item[(1)]
$T^H$ admits a balanced dualizing complex. 
\item[(2)]
$T^H$ is $d$-Cohen--Macaulay where $d=\injdim T$.
\end{enumerate}
\end{lemma}

Finally, we recall a definition of \cite[Definition 1.4]{JZ}. 
For each finitely generated graded left $A$-module $M$, define
$$B_M(t)=\sum_{i} (-1)^i h_{H^i_{\fm}(M)}(t),$$
which we can view as an element of ${\mathbb Q}((t^{-1}))$. 
We say that $M$ is {\it rational over ${\mathbb Q}$} if it 
satisfies the conditions:
\begin{enumerate}
\item[(a)]
$h_{M}(t)$ and $B_{M}(t)$ are rational functions over
${\mathbb Q}$ (inside ${\mathbb Q}((t))$ and 
${\mathbb Q}((t^{-1}))$ respectively), and 
\item[(b)]
as rational functions over ${\mathbb Q}$, we have 
\begin{equation}
\label{E2.6.1}\tag{E2.6.1}
h_{M}(t)=B_{M}(t).
\end{equation}
\end{enumerate}

Rationality over ${\mathbb Q}$ holds automatically for many 
graded algebras such as PI algebras and factor rings of AS regular
algebras \cite[Proposition 5.5]{JZ}. For simplicity, we assume

\begin{hypothesis}
\label{xxhyp2.7}
Every finitely generated graded left and right $A$-module is 
rational over ${\mathbb Q}$. 
\end{hypothesis}

\begin{conjecture}
\label{xxcon2.8} 
Every noetherian connected graded algebra with 
balanced dualizing complex satisfies Hypothesis 
\ref{xxhyp2.7}.
\end{conjecture}

No counterexample to the above conjecture is known.

\subsection{Castelnuovo--Mumford regularity}
\label{xxsec2.2}
In this subsection we recall the definitions and some basic 
properties of Castelnuovo--Mumford regularity and 
Ext-regularity in the noncommutative setting.
Noncommutative Castelnuovo--Mumford regularity was first 
studied by J{\o}rgensen in \cite{Jo2, Jo3} and later 
by Dong and Wu \cite{DW}.

\begin{definition}[{\cite[Definition 2.1]{Jo2}, \cite[Definition 4.1]{DW}}]
\label{xxdef2.9}
Let $M$ be a nonzero graded left $A$-module. The 
{\it Castelnuovo--Mumford regularity} (or \emph{CM regularity}, for short) of $M$ is defined to be
$$\begin{aligned}
\CMreg (M)&=\inf \{ p\in {\mathbb Z} \; \mid \; 
H^{i}_{\fm}(M)_{> p-i}=0, \forall \; i\in {\mathbb Z}\}\\
&=\sup \{i+\deg (H^{i}_{\fm}(M)) \; \mid  i\in {\mathbb Z}\}.
\end{aligned}$$
\end{definition}

As noted in \cite[Observation 2.3]{Jo3}, by \cite[Corollary 4.8]{VdB} 
$H^i_{\fm}(A) = H^i_{\fm^{\rm op}}(A)$ and hence 
$\CMreg({_AA}) = \CMreg ({A_A})$, which is simply denoted by 
$\CMreg(A)$.  Also by \cite[Theorem 8.3(3)]{AZ} if $B$ is a noetherian 
subring of $A$ and $A$ is finitely generated over $B$ on both sides then 
$\CMreg(A) =\CMreg({A_B})=\CMreg({_BA})$.

When $A$ has a balanced dualizing complex (for example, if $A$ is 
commutative or PI), $\CMreg(M)$ is finite for every nonzero 
finitely generated left graded $A$-module $M$ 
\cite[Observation 2.3]{Jo3}. But if $A$ does not have a balanced 
dualizing complex, then $\CMreg(A)$ could be infinite. For example, 
let $R_q$ be the noetherian connected graded algebra given in 
\cite[Theorem 2.3]{SZ}. It follows from \cite[Theorem 2.3(b)]{SZ} 
that $\deg H^1_{\fm}(R_q)=+\infty$ and therefore $\CMreg(R_q)=+\infty$.  
Next we give some examples where $\CMreg(M)$ is finite.

\begin{example}
\label{xxex2.10} The following examples are clear.
\begin{enumerate}
\item[(1)]
If $M$ is a finite-dimensional nonzero graded left $A$-module, 
then 
\begin{equation}
\label{E2.10.1}\tag{E2.10.1}
\CMreg(M)=\deg (M).
\end{equation}
A more general case is considered in part (4).
\item[(2)] \cite[Lemma 4.8]{DW}
Let $T$ be an AS Gorenstein algebra of type $(d, \bfl)$. Then
$\CMreg (T) =d-\bfl$. 
\item[(3)]
Let $T$ be an AS regular algebra of type $(d, \bfl)$. Then
$$\CMreg (T) =d-\bfl=\gldim T+\deg_t h_T(t),$$
which is a non-positive integer, see 
\cite[Proposition 3.1.4]{StZ}. It is easy to see 
that $T$ is Koszul if and only if $\CMreg(T)=0$
\cite[Proposition 3.1.5]{StZ}.
\item[(4)]
If $M$ is $s$-Cohen--Macaulay, then, by definition,
\begin{equation}
\label{E2.10.2}\tag{E2.10.2}
\CMreg(M)=s+ \deg(H^s_{\fm}(M)).
\end{equation}
\end{enumerate}
\end{example}

\begin{definition}[{\cite[Definition 2.2]{Jo3}, \cite[Definition 4.4]{DW}}]
\label{xxdef2.11}
Let $M$ be a nonzero graded left $A$-module. The 
{\it Ext-regularity} of $M$ is defined to be
\begin{align*}
\Extreg (M)&=\inf \{ p\in {\mathbb Z} \; \mid \; 
\Ext^i_A(M,\Bbbk)_{< -p-i}=0, \forall \; i\in {\mathbb Z}\}\\
&= - \inf \{ i+ \ged (\Ext^i_A(M, \Bbbk)) 
\; \mid \;  i\in {\mathbb Z}\}.
\end{align*}
The {\it Tor-regularity} of $M$ is defined to be
$$\begin{aligned}
\Torreg (M)&=\inf \{ p\in {\mathbb Z} \; \mid \; 
\Tor^A_i(\Bbbk, M)_{> p+i}=0, \forall \; i\in {\mathbb Z}\}\\
&=\sup\{ -i+\deg (\Tor^A_i(\Bbbk, M)) \; \mid \;
 i\in {\mathbb Z}\}.
\end{aligned}$$
If $M$ is a finitely generated graded module over a left 
noetherian ring $A$, then $\Extreg(M)=\Torreg(M)$ 
\cite[Remark 4.5]{DW}, and we will not distinguish between 
$\Extreg(M)$ and $\Torreg(M)$ in this case.
\end{definition}

\begin{example}
\label{xxex2.12}
The following examples are clear.
\begin{enumerate}
\item[(1)]
If $r=\Torreg(M)$, then 
\begin{equation}
\label{E2.12.1}\tag{E2.12.1}
t^A_i(_A M):=\deg(\Tor^A_i(\Bbbk, M))\leq r+i
\end{equation}
for all $i$.
\item[(2)] \cite[Example 4.6]{DW}
$\Extreg(A)=0$.
\end{enumerate}
\end{example}

Now we are ready to state some nice results 
in \cite{Jo2, Jo3, DW}.

\begin{theorem} 
\label{xxthm2.13}
Let $A$ be a noetherian connected graded algebra with a 
balanced dualizing complex and let $M\neq 0$ be a finitely 
generated graded left $A$-module. 
\begin{enumerate}
\item[(1)] \cite[Theorems 2.5 and 2.6]{Jo3} 
$$-\CMreg(A) \leq \Extreg(M)-\CMreg(M) \leq \Extreg(\Bbbk).$$
\item[(2)]  \cite[Corollary 2.8]{Jo3} \cite[Theorem 5.4]{DW}
$A$ is a Koszul AS regular algebra if and only if
$\Extreg(M)=\CMreg(M)$ for all $M$.
\item[(3)] \cite[Proposition 5.6]{DW}
If $M$ has finite projective dimension, then 
$$-\CMreg(A) = \Extreg(M)-\CMreg(M).$$
\end{enumerate}
\end{theorem}

In the next lemma, $\reg$ can be $\CMreg$, $\Extreg$, or
$\Torreg$. 

\begin{lemma}[{\cite[Corollary 20.19]{Ei}, \cite[Lemma 1.6]{Hoa}}]
\label{xxlem2.14}
Let $$0\to L\to M \to N\to 0$$ be a short exact sequence 
of finitely generated graded left modules over a 
left noetherian algebra $A$. 
\begin{enumerate}
\item[(1)]
$\reg(L) \leq \max \{\reg(M), \reg(N) + 1\}$ with 
equality if $\reg(M) \neq \reg(N)$.
\item[(2)]
$\reg(M) \leq \max\{ \reg(L), \reg(N)\}$. 
\item[(3)]
$\reg(N) \leq \max\{ \reg(L)-1, \reg(M)\}$ with 
equality if $\reg(M) \neq \reg(L)$.
\end{enumerate}
\end{lemma}

\begin{lemma}
\label{xxlem2.15}
Assume Hypothesis \ref{xxhyp0.3}. 
\begin{enumerate}
\item[(1)]
$\CMreg(T^H)\leq \CMreg(T)\leq 0$.
\item[(2)]
$\CMreg(T^H)=\CMreg(T)$ if and only if the $H$-action 
on $T$ has trivial homological determinant.
\item[(3)]
If both $T$ and $T^H$ are AS regular and $T^H\neq T$, 
then 
$$\CMreg(T^H)<\CMreg(T).$$ 
\end{enumerate}
\end{lemma}

\begin{proof}
(1) By Example \ref{xxex2.10}(3), $\CMreg(T)\leq 0$.
So it suffices to show the first inequality.
By \cite[Lemma 2.5(b)]{KKZ3}, $H^i_{\fm}(T^H)$ is a 
direct summand of $H^i_{\fm}(T)$ for all $i$. 
In particular, if $d$ is the injective dimension of $T$, then
$\deg (H^d_{\fm}(T^H))\leq \deg (H^d_{\fm}(T))$ 
and $H^i_{\fm}(T^H)=H^i_{\fm}(T)=0$ for all
$i\neq d$. 
The assertion now follows from Definition 
\ref{xxdef2.9}.

(2) By the proof of part (1), one sees that 
$\CMreg(T^H)=\CMreg(T)$ if and only if 
$\deg (H^d_{\fm}(T^H))=\deg (H^d_{\fm}(T))$,
and if and only if $\ged (H^d_{\fm}(T^H)')=
\ged (H^d_{\fm}(T)')$. Now the assertion basically 
follows from \cite[Lemma 3.5(f)]{KKZ3} (after one matches
up the notation). 

(3) This follows from \cite[Theorem 0.6]{CKWZ1} and 
part (2).
\end{proof}

\begin{lemma}
\label{xxlem2.16} 
Let $f:S\to T$ be a graded algebra homomorphism between 
two noetherian AS regular algebras such that $T$ 
is finitely generated over $S$ on both sides. Let 
$\delta(T/S)= \CMreg(T) -\CMreg(S)$.
\begin{enumerate}
\item[(1)]
$t^S_i(T_S) \leq \delta(T/S)+i$ for all $i\geq 0$.
\item[(2)]
$\deg T/TS_{\geq 1}\leq \delta(T/S)$.
\item[(3)]
Assume Hypothesis \ref{xxhyp0.3} and, in addition,  that
the image of $f$ is $T^H$. Then 
$\tau_H(T)\leq \delta(T/S)+1$. Similarly,
$\tau^{op}_H(T)\leq \delta(T/S)+1$.
\end{enumerate}
\end{lemma}

\begin{proof}
(1) By Example \ref{xxex2.12}(1) and Theorem \ref{xxthm2.13}(3),
we have
$$t^S_i(T_S)\leq \Torreg(T_S)+i=
\CMreg(T_S)-\CMreg(S)+i=
\CMreg(T)-\CMreg(S)+i.$$

(2) By definition,
$$\deg T/TS_{\geq 1}=\deg \Tor^S_0(T,\Bbbk)
=t^S_0(T_S)\leq \delta(T/S).$$

(3) By definition, $\tau_H(T)=\deg T/TS_{\geq 1}+1$.
The assertion follows from part (2).
\end{proof}

\section{$\tau$-saturation degree and $\beta(A)$}
\label{xxsec3}

In this section, we generalize some arguments in commutative
invariant theory (see \cite{De, DS, Fl, Fo, Ga}) to a 
noncommutative setting. As in Hypothesis \ref{xxhyp0.3}, 
we usually assume that the Hopf algebra $H$ is semisimple. 
In the commutative setting, if the group $G$ acts on 
$\Bbbk[x_1, \dots, x_n]$ via graded automorphisms, then 
\begin{equation}
\label{E3.0.1}\tag{E3.0.1}
\beta(\Bbbk[x_1, \dots, x_n]^G) \leq 
\tau_G(\Bbbk[x_1, \dots, x_n]) \leq |G|
\end{equation}
\cite{De, Fo}. Similarly, if $T$ is noncommutative and $H$ 
is a semisimple Hopf algebra acting homogeneously on $T$, 
we relate $\tau_H(T)$ to $\beta(T^H)$. As noted in Question 
\ref{xxque0.4}, in the noncommutative case, we do not have 
a general upper bound on $\tau_H(T)$; it would be nice to 
have such a bound even for particular classes of $T$ and $H$.

\subsection{Easy observations}
\label{xxsec3.1}
Let $A$ be a connected graded algebra with maximal graded 
ideal $\fm$. If $\fm$ is generated by a set of homogeneous 
elements as a left (or, right) ideal, then $A$ is generated 
by the same set as an algebra. Another way of saying this  is 
that $A$ is generated by $\fm/\fm^2$, as it is well-known that
\begin{equation}
\label{E3.0.2}\tag{E3.0.2}
\fm/\fm^2\cong \Tor^A_1(\Bbbk,\Bbbk)
\end{equation}
as graded vector spaces. 

If $M=\bigoplus_{i\in {\mathbb Z}} M_i$ is a 
${\mathbb Z}$-graded vector space, then let $M_{\leq n}$ denote 
the subspace $\bigoplus_{i\leq n} M_i$. The following lemma is 
easy to prove, and the proof is omitted.

\begin{lemma}
\label{xxlem3.1}
Let $A$ be a connected graded algebra and let $H$ be a 
semisimple Hopf algebra acting on $A$ homogeneously. Let $B$ 
be a factor graded ring of $A$ with induced $H$-action. Let 
$d$ be a positive integer. Then the following hold.
\begin{enumerate}
\item[(1)]
$\beta(B)\leq \beta(A)$. 
\item[(2)]
If $B=A/A_{\geq d}$, then
$\beta(B)=\deg (\Tor^A_1(\Bbbk,\Bbbk)_{\leq d-1})$.
\item[(3)]
$B^H$ is a factor ring of $A^H$.
\item[(4)]
$(A/A_{\geq d})^H=A^H/(A^H)_{\geq d}$.
As a consequence, 
$$\beta((A/A_{\geq d})^H)=
\deg (\Tor^{A^H}_1(\Bbbk,\Bbbk)_{\leq d-1}).$$
\end{enumerate}
\end{lemma}

The main result of the next lemma is essentially the 
same as \cite[Lemma VI.7]{Ga}. Here we prove a slightly 
more general result.

\begin{lemma}
\label{xxlem3.2}
Suppose that $f: B \to A$ is a graded algebra homomorphism of 
connected graded algebras with $A$ generated in degree $1$ and 
that $C := \im(f)$ is a direct summand of $A$ as a right $B$-module.
\begin{enumerate}
\item[(1)] 
The left ideal $J = A C_{\geq 1}$ is generated by a set of 
elements of degree $\leq t_0^B(A_B) + 1$.
\item[(2)] 
If $J$ is generated by elements in $C_{\geq 1}$ of degree $\leq r$, 
then $\beta(C) \leq r$.
\item[(3)] 
$\beta(C) \leq t_0^B(A_B) + 1$.
\end{enumerate}
The analogous results hold for the right ideal $C_{\geq 1} A$ and 
$t_0^B({}_B A)$.
\end{lemma}

\begin{proof} 
(1) Let $g= t_0^B(A_B) + 1$. Since 
\[ \Tor_0^B(A_B, \Bbbk) = A \otimes_B \Bbbk 
= A \otimes_B B/B_{\geq 1} = A/AB_{\geq 1}=A/AC_{\geq 1}=A/J
\]
is equal to zero in degrees $\geq g$, therefore 
$J\supseteq A_{\geq g}$. Every element in $J$ of degree $>g$ is 
generated by $A_{g}$ since $A$ is generated in degree 1. Therefore
$J$ is generated by $\bigoplus_{i=1}^{g} J_i$ as a left ideal 
of $A$. Let $\{f_1,\cdots,f_r\}$ be a basis of
$\bigoplus_{i=1}^{g} J_i$, and write 
$$f_i=\sum_{j=1}^n h_{ij} u_j, \quad h_{ij}\in A, 1\leq i\leq r$$
where $u_1,\cdots,u_n$ are in $C_{\geq 1}$. We may assume that
all elements are homogeneous and for each $u_j$ there is at least
one $i$ such that $h_{ij} u_j\neq 0$. By removing all terms 
which can be canceled, we have that $\deg f_i=\deg h_{ij} \deg u_j$
which shows that $\deg u_i\leq \deg f_i\leq g$. It is clear
that $J$ is generated by $u_1,\cdots,u_n$ as desired.

(2) Suppose that $J$ is generated as a left ideal of $A$ by 
$u_1,\cdots, u_n\in C_{\geq 1}$ where $\deg u_i \leq r$ for all 
$i$. We claim that $C_{\geq 1}$ is generated as a left ideal of 
$C$ by $u_1,\cdots, u_n$. Let $f\in C_{\geq 1}\subseteq J$. Then 
$$f=g_1 u_1+\cdots+ g_n u_n$$
where each $g_i\in A$. By hypothesis, $C$ is a direct summand 
of $A$ as a right $B$-module. Write
\[ A_B = C_B \oplus D_B.\]
Then each $g_i$ can be written in the form 
$g_i = \gamma_i + \gamma'_i$ where $\gamma_i \in C$ and 
$\gamma'_i \in D$. Hence
\[ f = (\gamma_1 u_1 + \dots + \gamma_n u_n ) + 
(\gamma'_1 u_1 + \dots + \gamma'_n u_n)\]
and since $f \in C$, 
\[ f = \gamma_1 u_1 + \dots + \gamma_n u_n.
\]
Therefore, $C_{\geq 1}$ is generated by $\{u_i\}$ as a left ideal 
of $C$. It follows that $C$ is generated as an algebra by 
$\{u_i\}$ so $\beta(C) \leq r$.

(3) This is an immediate consequence of parts (1) and (2).
\end{proof}

\begin{corollary}
\label{xxcor3.3}
Let $A$ be a connected graded algebra and let $H$ be a 
semisimple Hopf algebra acting on $A$ homogeneously. Then 
$\beta(A^H) \leq \tau_H(A)$. 
\end{corollary}

\begin{proof}
This follows from Lemma~\ref{xxlem3.2} by taking $B = C = A^H$ 
and $f$ to be the natural inclusion into $A$. 
\end{proof}

In many cases $\tau_H(A)$ is easier to compute than $\beta(A^H)$. 
Furthermore the bound can be sharp.

\begin{example}
\label{xxex3.4}
Let $T$ be the down-up algebra 
$$T: = A(0,1)= \Bbbk\langle x,y \rangle/
(x^2y = yx^2, xy^2 = y^2x ),$$ 
which is an AS regular algebra of dimension 3.  Let $\sigma$ be 
the automorphism of $T$ defined by $\sigma(x)=-x$ and $\sigma(y)=y$.
Then $\sigma$ generates a group $G$ of order 2 that acts on $T$. 
One can check that a basis for the invariants of degree $\leq 3$ 
is given by:
$$y \; (\deg 1), \quad
x^2,y^2 \; (\deg 2), \quad
x^2y(=yx^2), xyx, y^3 \; (\deg 3).$$
Taking left (respectively, right) multiples of these 
invariants and determining the dimension of $J_G(T)_3$, 
it is not hard to check that $\tau_G(T) = \tau_G^{\op}(T) = 3$, 
and so by Corollary~\ref{xxcor3.3}, $\beta(T^G) \leq 3.$

One can also check that the invariant $xyx$ of degree $3$ is 
not generated by the lower degree invariants, so that 
$\beta(T^G) = 3= \tau_G(T)$.
\end{example}

As a corollary of Lemma~\ref{xxlem3.2}, we obtain the 
following degree bound.

\begin{theorem}
\label{xxthm3.5}
Suppose that $(T,H)$ satisfies Hypothesis \ref{xxhyp0.3}. 
Assume further
\begin{enumerate}
\item[(a)]
$T$ is a noetherian AS regular domain generated 
in degree 1 such that $T\# H$ is prime.
\item[(b)]
$T^H$ has finite global dimension.
\end{enumerate}
Then the following hold. 
\begin{enumerate}
\item[(1)]
$\beta(T^H)\leq \dim H$.
\item[(2)]
Suppose $\beta(T^H)=\dim H$ and that $\Bbbk$ is algebraically 
closed of characteristic 0. Then $H=\Bbbk G$ where
$G={\mathbb Z}/(d)$ for $d=\dim H$, the $G$-action on $T$ is 
faithful, and $G$ is generated by a quasi-reflection of $T$ 
in the sense of \cite[Definition 2.2]{KKZ1}.
\end{enumerate}
\end{theorem}

\begin{proof}
(1) This follows from Example~\ref{xxex1.6} and 
Corollary~\ref{xxcor3.3}.

(2) By Proposition \ref{xxpro1.5}, $T$ is free over $R:=T^H$
of rank $d:=\dim H$; moreover, the cyclic graded left $T$-module 
$M:=T \otimes_R (R/R_{\geq1})$ has $\Bbbk$-dimension $d$ and 
$M_i \neq 0 $ for $i = 0, \dots, \tau_H(T) -1$. But by 
Corollary~\ref{xxcor3.3} and Example~\ref{xxex1.6} we have 
$\dim H = \beta(R) \leq \tau_H(T) \leq \dim H$, so 
$\tau_H(T) = \dim H$. Hence $\deg M=d-1$ and the Hilbert series 
of $M$ is $1+t+t^2+\cdots +t^{d-1}$, or equivalently,
$\dim (T R_{\geq 1})_i = \dim T_i - 1$ for all $0 \leq i \leq d-1$.

As a result, as an $H$-representation, $T_1 = (T^H)_1 \oplus N$ 
where $N$ is a one-dimensional simple $H$-module, and so $T_1$ 
is the direct sum of one-dimensional $H$-modules. Hence, 
$[H,H] \cdot T_1 = 0$ (where $[H,H]$ is the commutator ideal of
$H$) and since $[H,H]$ is a Hopf ideal and $T$ is generated in 
degree 1, we have $[H,H] \cdot T = 0$. Therefore, the quotient 
Hopf algebra $\overline{H} = H/[H,H]$ acts on $T$. By part (1), 
$$\dim(H) = \beta(T^H) = \beta(T^{\overline{H}})  
\leq \dim(\overline{H}),$$ 
so $\dim(\overline{H}) = \dim (H)$. Therefore, $[H,H] = 0$ and 
$H$ is commutative.

Since $H$ is semisimple and commutative, it is the dual of a 
group algebra $H = (\Bbbk K)^*$ for some finite group $K$. 
Therefore, $\Bbbk K$ coacts on $T$ and the direct sum 
decomposition $T_1 = (T^H)_1 \oplus N$ is also a decomposition 
as $\Bbbk K$-comodules. By \cite[Theorem 3.5(3,4)]{KKZ4}, 
$K$ is generated by the $K$-grading of $N$, so $K = \mathbb{Z}/(d)$ 
is a cyclic group of order $d$.  

Let $x$ be a basis element of $N$ and let $\tau$ be the $K$-degree 
of $x$.  By \cite[Theorem 3.5(1)]{KKZ4}, $x=f_{\tau}$ which is a 
generator of the $T^H$-module $T_{\tau}$ (in the notation of 
\cite{KKZ4}). Since $T_1 = (T^H)_1 \oplus N=(T^H)_1 \oplus \Bbbk x$, by 
\cite[Theorem 3.5(1)]{KKZ4}, $x$ is a normal element in $T$. Further, 
for each $0\leq i\leq d-1$, $x^i$ equals $f_{\tau^i}$ which is a 
generator of $T_{\tau^i}$.
By \cite[Theorem 3.5(2)]{KKZ4}, $TR_{\geq 1}$ is a 2-sided ideal
of $T$ and $M$ is isomorphic to the graded algebra
$\Bbbk[x]/(x^d)$ with an induced $K$-comodule structure such that
$M$ is a free $\Bbbk K$-comodule of rank one. As a consequence,
$T = \bigoplus_{i=0}^{d-1} x^i T^H=\bigoplus_{i=0}^{d-1} T^H x^i$.
Write $h_T(t)=\frac{1}{(1-t)^n p(t)}$ and $h_{T^H}(t)=\frac{1}{(1-t)^nq(t)}$
where $n=\GKdim T=\GKdim T^H>0$ 
and $p(1), q(1) \neq 0$ (see \cite[Definition 2.2]{KKZ1}). 
Then 
\begin{align*}
\label{E3.5.1}\tag{E3.5.1}
h_{T}(t) &=\frac{1}{(1-t)^n p(t)}=(1+t+\cdots+t^{d-1}) h_{T^H}(t)\\
&=\frac{1+t+\cdots+t^{d-1}}{(1-t)^n q(t)}=\frac{(1-t^d)}{(1-t)^{n+1} q(t)}.
\end{align*}

Since $\Bbbk$ is algebraically closed of characteristic 0, 
$(\Bbbk K)^*$ is isomorphic to $\Bbbk G$ where $G$ is a group 
that is also isomorphic to $\mathbb{Z}/(d)$. Since 
the $\Bbbk K$-coaction on $M$ is free (hence faithful),
the $\Bbbk G$-action on $M$ is free (hence faithful).
As a consequence, the $\Bbbk G$-action on $T$ is faithful.

Viewing $H = \Bbbk G$, suppose that $G$ is generated by the 
automorphism $\sigma$ of $T$. Examining the action of $\sigma$ 
on $T_1 = (T^H)_1 \oplus \Bbbk x$, we have $\sigma(x) = \lambda x$ where 
$ \lambda$ is a primitive $d$th root of unity. As a right $T^H$-module, 
$T = \bigoplus_{i=0}^{d-1} x^i T^H$ and $\sigma$ acts on $x^i T^H$ 
by scaling by $ \lambda^i$. Then 
$$\begin{aligned}
{\text{Tr}}_T(\sigma)&\;\;\;\;=\;\;\;\;
\sum_{i=0}^{d-1}  \lambda^i t^i h_{T^H}(t)=\frac{\sum_{i=0}^{d-1}  \lambda^i t^i}
{(1-t)^n q(t)}\\
&\;\;\;\;=\;\;\;\; \frac{(1-( \lambda t)^d)}{(1- \lambda t) (1-t)^n q(t)}=\frac{(1-t^d) (1-t)}
{(1- \lambda t)(1-t)^{n+1} q(t)}\\
&\overset{\eqref{E3.5.1}}{=}
\frac{(1-t)}{(1- \lambda t)(1-t)^{n} p(t)} \\
&\;\;\;\; =\;\;\;\; \frac{1}{(1-t)^{n-1}(1- \lambda t) p(t)}.
\end{aligned}
$$
By \cite[Definition 2.2]{KKZ1},
$\sigma$ is a quasi-reflection.
\end{proof}

A bound on $\beta(T^H)$ would be useful in projects
such as  \cite{FKMW1, FKMW2}, where $T^H$ has finite 
global dimension and the generators of $T^H$ were 
determined explicitly.

We now give a family of examples of AS regular algebras 
$T$ and groups $G$ that show that $\beta(T^G)$ can be 
arbitrarily larger than $|G| = \dim \Bbbk G$.

\begin{example}
\label{xxex3.6}
Let $m$ be a fixed positive integer larger than 2. Let 
${\mathfrak g}$ be the free Lie algebra generated by
$x$ and $y$. We consider ${\mathfrak g}$ as a graded Lie
algebra. The universal enveloping algebra
$S:=U({\mathfrak g})$ is isomorphic to the free algebra
$\Bbbk \langle x,y \rangle$ as graded algebras by assigning 
$\deg x=\deg y=1$.

Consider the graded quotient Lie algebra
$${\mathfrak g}_m={\mathfrak g}/{\mathfrak g}_{\geq m}.$$
Namely, ${\mathfrak g}_m$ is a graded 
Lie algebra generated by $x,y$ and
subject to relations 
$$[a_1,[a_2,[\cdots, [a_{m-1},a_m]\cdots]]]=0$$
for all $a_i$ being $x$ or $y$. 
Then $\mathfrak g$ is a finite-dimensional graded Lie algebra. 
Its universal enveloping algebra
$T = U({\mathfrak g}_{m})$ is a noncommutative 
noetherian AS regular algebra. 

There is a natural surjective Lie algebra map ${\mathfrak g}
\to {\mathfrak g}_m$ that induces a surjective graded algebra 
map $\phi: S \to T$. Since every 
relation in $T$ has degree (at least) $m$, we have that
$\phi$ is a bijective map for degree less than $m$. 
Hence $S/S_{\geq m}=T/T_{\geq m}$. 

Now let $\sigma$ be the automorphism of ${\mathfrak g}$
sending $x$ to $-x$ and $y$ to $y$. Then $\sigma$ has order
2, and it induces order 2 graded algebra automorphisms
on $S$, $T$, $S/S_{\geq m}$ and $T/T_{\geq m}$.
Let $G=\langle \sigma\rangle$. We claim that
$\beta(T^G)\geq m-1$. To see this, we use 
Example \ref{xxex1.3} and Lemma \ref{xxlem3.1} as below.
$$\begin{aligned}
\beta(T^G) &\overset{{\text{Lem \ref{xxlem3.1}(1)}}}{\geq}
\beta (T^G/(T^G)_{\geq m})   \\
&\overset{{\text{Lem \ref{xxlem3.1}(4)}}}{=}
\beta((T/T_{\geq m})^G) =\beta((S/S_{\geq m})^G)\\
&\overset{{\text{Lem \ref{xxlem3.1}(4)}}}{=}
\beta(S^G/(S^G)_{\geq m}) \\
&\overset{{\text{Lem \ref{xxlem3.1}(4)}}}{=}
\deg (\Tor^{S^G}_1(\Bbbk,\Bbbk)_{\leq m-1})\\
&\overset{{\text{Exam \ref{xxex1.3}}}}{=}
m-1. \\
\end{aligned}
$$

Since $m$ can be arbitrarily large, there is no uniform bound 
for $\beta(T^G)$ which depends only on $|G|$ when we consider 
all noetherian AS regular algebras $T$.
\end{example}

\subsection{Using central subrings to bound $\beta(R)$}
\label{xxsec3.2}
In this subsection, we obtain bounds on degrees of invariants 
for certain noncommutative rings by leveraging central 
subrings. We first note that unlike group actions, Hopf actions need not
preserve the center of $T$.

\begin{example}
\label{xxex3.7}
Let $H$ be the eight-dimensional Kac--Palyutkin semisimple Hopf 
algebra. As an algebra, $H$ is generated by $x$, $y$, and $z$ 
subject to the relations
\[ x^2 = y^2 = 1, xy = yx, zx = yz, zy = xz, z^2 
= \frac{1}{2}(1 + x + y -xy)
\]
with coalgebra structure given by
\[ \Delta(x) = x \otimes x, \Delta(y) = y \otimes y, \Delta(z) 
= \frac{1}{2}(1 \otimes 1 + 1 \otimes x 
+ y \otimes 1 - y \otimes x)(z \otimes z)
\]
and
\[ \varepsilon(x) = \varepsilon(y) = \varepsilon(z) = 1.
\]
Let $$T:= \frac{\Bbbk \langle u,v \rangle}{(u^2-v^2).}$$  
This AS regular algebra $T$ is isomorphic to $\Bbbk_{-1}[x_1,x_2]$ 
and  has a basis of elements of the form $u^i(vu)^jv^k$ where 
$i,j \in \mathbb{N}$ and $k=0,1$. By \cite[Section 2]{FKMW1} 
$H$ acts on $T$  with the actions of the generators $x,y,z$ 
of $H$ satisfying
$$x.uv=-uv,\quad x.vu= -vu \quad y.uv= -uv \quad y.vu = -vu$$ 
$$ z.uv=-vu \quad z.vu = uv.$$ 
This action does not preserve the center of $T$, as $u^2$ 
and $uv+vu$ are central, but $z.(uv+vu) = -vu+uv$ is not. 
\end{example}

Our first result uses a noncommutative generalization of 
Broer's bound. Broer's upper bound on $\beta(A^G)$ when 
$A:= \Bbbk[x_1, \dots, x_n]$  (proved in \cite{Br} 
when $A^G$ is Cohen--Macaulay and extended to the modular 
case in \cite[Corollary 0.3] {Sy})  was generalized in 
\cite{KKZ6} to quantum polynomial algebras (namely, noetherian 
AS regular domains with global dimension $n$ and Hilbert 
series $1/(1-t)^n$ for some $n$) and used to obtain bounds 
on minimal generators of  $\Bbbk_{-1}[x_1, \dots, x_n]^G$ 
for permutation representations $G$ (providing an analogue 
of G\"{o}bel's theorem \cite{Go}). We can use this result 
to obtain a bound for any group $G$ acting on 
$\Bbbk_{-1}[x_1, \dots, x_n]$.

\begin{lemma}[Broer's Bound, {\cite[Lemma 2.2]{KKZ6}}]
\label{xxlem3.8}
Let ${T}$ be a quantum polynomial algebra of dimension $n$, 
$H$ be a semisimple Hopf algebra. Suppose that 
$C \subset {T}^{H} \subset {T}$, for some graded iterated 
Ore extension 
$C = \Bbbk[f_1][f_2: \tau_2, \delta_2] \dots 
[f_n: \tau_n, \delta_n]$ such that ${T}_C$ is finitely 
generated, and $\deg f_i > 1$ for at least 2 distinct $i$'s. 
Then
$$\beta({T}^{H}) \leq \sum \deg{f_i} - n.$$
\end{lemma}

For a more general result, see \cite[Lemma 2.1]{KKZ6}.

\begin{corollary}
\label{xxcor3.9}
Let $\Bbbk$ be an infinite field and $G$ be a finite group acting 
as graded automorphisms on $\Bbbk_{-1}[x_1, \dots, x_n]$. Then 
$$\beta(\Bbbk_{-1}[x_1, \dots, x_n]^G) \leq n(2|G| -1).$$
\end{corollary}

\begin{proof}
The group $G$ acts on the commutative polynomial ring 
$S:= \Bbbk[x_1^2, \dots, x_n^2]$. 
By Lemma \ref{xxlem3.10} below, $\beta(C) \leq 2|G|$
for a subring of primary invariants $C\subseteq S^G$, which is 
isomorphic to a polynomial ring. 
The bound follows from Lemma \ref{xxlem3.8}.
\end{proof}

We will have an improvement to this result in Corollary \ref{xxcor3.12}.
The following lemma is due to Dade, see \cite[p.483]{St}.

\begin{lemma}\cite[Proposition 3.4]{St}
\label{xxlem3.10}
Let $G$ be a finite group acting linearly on $\Bbbk[x_1, \dots, x_n]$
with $\deg x_i=1$ for all $i$. Assume that $\Bbbk$ is an infinite
field. Then there are $n$ primary invariants 
that have degree $\leq |G|$.
\end{lemma}

Note that the proof of \cite[Proposition 3.4]{St} uses  only
the fact that $\Bbbk$ is infinite, not the hypothesis that
$\ch \Bbbk=0$.

Our second result in this subsection applies to algebras which 
are module-finite over their centers. The first step in Fogarty's 
proof of the Noether bound in the non-modular case \cite{Fo} 
shows that the product of any $|G|$ elements of degree 1 in 
$A:=\Bbbk[x_1, \dots, x_n]$ is contained in the Hilbert ideal 
$J_G(A)$ (i.e. $\tau_G(A) \leq |G|$).  This result is proved by 
indexing these $G$ elements as $f_\alpha$ for $\alpha \in G$, 
and using the identity (sometimes called Benson's Lemma)
$$ \sum_{ \emptyset \neq S \subseteq G}(-1)^{|G-S|}
\sum_{\tau \in G} \prod_{\alpha\in S} (\tau \alpha f_\alpha )
\prod_{ \alpha \in G - S} f_\alpha \quad \in \quad J_G(A)$$
where the leftmost sum is taken over all nonempty subsets 
$S \subseteq G$. For an algebra $A$ which is module-finite over 
its center, we can use this idea of Fogarty to obtain the 
following result.

\begin{proposition}
\label{xxpro3.11}
Let $A$ be a connected graded domain and let $H$ be a 
semisimple Hopf algebra acting on $A$ homogeneously.
Suppose that
\begin{enumerate}
\item[(a)]
$A$ is a finitely generated module over a central subalgebra 
$Z$;
\item[(b)]
$Z$ is stable under the $H$-action;
\item[(c)]
$Z$ is generated as an algebra by elements of degree $\leq d$; 
\item[(d)]
$A$ is generated as a $Z$-module by elements of degree $\leq m$; and 
\item[(e)]
either $\Bbbk$ is an algebraically closed field of characteristic zero 
or $H$ is a group algebra. 
\end{enumerate}
Then
\begin{enumerate}
\item[(1)] $\beta(A^H) \leq \tau_H(A)\leq d \dim H+ m.$
\item[(2)]
For the $i$th annihilator ideal $J_{H,i}$ defined in 
Definition~\ref{xxdef1.1}(3),
$$\deg A/J_{H,i}\leq \deg A/J_{\infty}\leq d \dim H + m -1.$$
\end{enumerate}
\end{proposition}

\begin{proof}
(1) First we claim that the $H$-action on $A$ induces a 
group action on $Z$. If $H$ is a group algebra, this is 
clear by Hypothesis (b). If $H$ is not a group algebra, then we 
assume that $\Bbbk$ is an algebraically closed field of 
characteristic zero. Since $Z$ is a commutative domain and $H$ 
acts on $Z$, by a result of Etingof--Walton \cite[Theorem 1.3]{EW},
the action of $H$ on $Z$ factors through a group action. Let this 
group algebra  be $\Bbbk G$, which is a quotient Hopf algebra of $H$. 

Let $g = |G|$, which is bounded by $\dim H$. Suppose that $A$ is 
generated as a module over $Z$ by homogeneous elements $x_1, \dots x_r$ 
of degree $\leq m$. By hypothesis, $G$ acts on the commutative algebra 
$Z$. We remark that since $H$ is semisimple, in the case that 
$\ch \Bbbk \neq 0$, $g = \dim H$ is invertible in $\Bbbk$. Hence, by 
Fogarty's proof \cite{Fo}, we have that all $g$-fold products of 
positive degree elements of $Z$ are contained in the left Hilbert 
ideal $J_G(Z)=J_H(Z)$ and so $J_H(Z)_i = Z_i$ for all $i \geq gd$ 
(that is, $\tau_G(Z) \leq gd$). Observe also that $J_H(Z) \subseteq 
J_H(A)$.

Now since $A = Zx_1 + Z x_2 + \dots + Z x_r$ and $\deg x_i \leq m$, 
every homogeneous element $a$ of degree at least $gd + m$ can be 
written as $a = \sum z_i x_i$, where the $z_i$ are homogeneous of 
degree at least $gd$. Since for all $i$, $z_i \in J_H(A)$, we have 
$a \in J_H(A)$. Therefore, $A_{\geq gd + m} \subseteq J_H(A)$, so 
$\tau_H(A) \leq gd +m$. By Corollary~\ref{xxcor3.3}, $\beta(A^H) 
\leq d |G| + m\leq d \dim H+m$, as desired.

(2) Since $Z^G$ is a central subalgebra of $A^H$, the actions of $Z^G$  
on the left and the right of $\Tor^{A^H}_i(A,\Bbbk)$ are the same. For every 
$i\geq gd$ and $z\in Z_i$, by the proof of part (1), we can write 
$z=\sum y_j f_j$, where $y_j\in (Z^G)_{\geq 1}$ and $f_j\in Z$. 
For each $j$, we have $y_j \cdot \Bbbk = 0$, and so 
$y_j \in J_{H,i}(A)$ for all $i$. Hence, $z \in J_{H,i}(A)$ so 
$Z_{\geq gd} \subseteq J_{H,i}(A)$. By the proof of part (1), 
$A_{\geq gd +m}$ is in the ideal generated by $Z_{\geq gd}$. 
Therefore $A_{\geq gd +m}$ is in $J_{H,i}(A)$, as desired.
\end{proof}

We remark that $A:=\Bbbk_{-1}[x_1, \dots, x_n]$ is a finite module 
over the central subalgebra $Z:= \Bbbk[x_1^2, \dots, x_n^2]$ 
and is generated as a $Z$-module by elements of degree $\leq n$. 
It is easy to see that every group action on $A$ induces an action
on $Z$. Hence, we have the following corollary, which is an 
improvement of the result of Corollary~\ref{xxcor3.9}.

\begin{corollary}
\label{xxcor3.12}
Let $G$ be a finite group acting as graded automorphisms on 
$\Bbbk_{-1}[x_1, \dots, x_n]$ and suppose that $|G|$ is invertible 
in $\Bbbk$. Then 
$$\beta(\Bbbk_{-1}[x_1, \dots, x_n]^G) \leq 2|G| + n.$$
\end{corollary}

Corollary \ref{xxcor3.12} suggests the following questions, which 
are subquestions of Question \ref{xxque0.4} (see also Questions 
\ref{xxque6.3} and  \ref{xxque6.4}).

\begin{question}
\label{xxque3.13}
Suppose $G$ is a finite group and $H$ is a semisimple Hopf algebra.
\begin{enumerate}
\item[(1)]
Is there an upper bound on $\beta(\Bbbk_{-1}[x_1, \dots, x_n]^G)$ 
that depends only upon $|G|$ and not on the dimension of the 
representation of $G$?
\item[(2)]
A Hopf action need not preserve the subring 
$S:= \Bbbk[x_1^2, \dots, x_n^2]$. If $H$ acts homogeneously 
on $\Bbbk_{-1}[x_1, \dots, x_n]$, is there an analogous bound for 
$\beta(\Bbbk_{-1}[x_1, \dots, x_n]^H)$?
\end{enumerate}
\end{question}

\section{Bounds on the degree of $\Tor_i^A(M,\Bbbk)$}
\label{xxsec4}

In the next two sections we will prove results which provide
bounds on $t^A_i(M) = \deg \Tor_i^A(M,\Bbbk)$ for various 
$A$-modules $M$. A special case is when either $A$ is $T^H$ 
(for $T$ and $H$ satisfying Hypothesis~\ref{xxhyp0.3}) or 
$A$ is a noetherian AS regular algebra mapping to $T^H$.
Recall from \eqref{E1.0.6} that we also use the notation
$$\beta_i(A)=t^A_i(\Bbbk)$$
where $\beta_1(A)=: \beta(A)$ is a bound for the degrees of 
the minimal generators of $A$ and $\beta_2(A)$ provides a 
bound for the degree of the relations of $A$.

In general, the connected graded algebra $R:= T^H$ need not 
be AS regular. However, when $R$ is commutative, the Noether 
Normalization Theorem states that there exists a nonnegative 
integer $d$ and algebraically independent homogeneous elements 
$y_1, y_2, \cdots, y_d$ in $R$ such that $R$ is a finitely 
generated module over the polynomial subring $S = 
\Bbbk[y_1, y_2, \cdots, y_d]$. When $R$ is noncommutative, 
such a result fails, in general, even when allowing $S$ to be 
a noncommutative AS regular algebra.

Nevertheless, if we suppose the existence of an AS regular 
subalgebra $S\subseteq R$, then we are able to prove bounds 
for $\beta(R)$ (as well as $\beta_i(R)$ for $i\geq 2$) by 
understanding the connection between $S$ and $R$. Therefore, 
some results in this section assume the existence of an AS 
regular version of a Noether normalization (e.g., as in 
Lemma \ref{xxlem3.8}) or a map from some AS regular algebra 
$S \rightarrow R$ such that $R$ is a finitely generated $S$-module.
In particular, we generalize some results of Symonds in \cite{Sy} 
and of Derksen in \cite{De}.

Throughout this section, we fix the following notation.

\begin{notation}
\label{xxnot4.1}
Let $A$ and $B$ be connected graded algebras, and let $f: A\to B$ 
be a graded algebra homomorphism making $B$ a finitely generated
graded left $A$-module generated by a set of homogeneous elements, 
say $\{v_i\}_{i \in S_1}$ including $1$, with degree no more than 
$t^A_0(_A B):=\deg \Tor^A_0(\Bbbk, B)$.  Let $A$ be generated as 
an algebra by a set of homogeneous elements, say 
$\{x_j\}_{j \in S_2}$, of degrees no more than $\beta(A)$.
\end{notation}

The next lemma is \cite[Proposition 2.1(1)]{Sy}.

\begin{lemma}
\label{xxlem4.2}  Assume Notation \ref{xxnot4.1}. Then 
$$\beta(B) \leq \max\{ \beta(A), t^A_0(_A B)\}.$$
\end{lemma}

\begin{proof} Write $B=\sum_{i\in S_1} A v_i$. Then, as an 
algebra, $B$ is generated by $\{v_i\}_{i\in S_1} 
\cup \{x_j\}_{j\in S_2}$. Hence the assertion follows.
\end{proof}

Next we generalize \cite[Proposition 2.1(2,3)]{Sy} which 
concerns bounds on the degrees of the relations in $A$, that is, $\beta_2(A)$.
In the noncommutative case, we can obtain only a weaker bound.

Let $A$ be a connected graded algebra and write $A$ as
$$A=\Bbbk \langle \mathcal{G}(A)\rangle/I(A)$$ 
where $\mathcal{G}(A)$ is a minimal generating set of $A$ and 
$I(A)$ is the two sided ideal of the relations in $A$. Here 
$\Bbbk \langle \mathcal{G}(A)\rangle$ is the free algebra 
generated by the graded vector space $\mathcal{G}(A)$. Let 
$N \geq \beta(A)$ be a positive integer. Define 
\[ \Phi_{N}(A) = \Bbbk \langle \mathcal{G}(A)\rangle /(I(A)_{\leq N})
\]
where $(I(A)_{\leq N})$ is the ideal of the free algebra 
$\Bbbk \langle \mathcal{G}(A)\rangle$ generated by $I(A)_{\leq N}$. By 
definition, there is a canonical surjective algebra map 
$\pi_A: \Phi_N(A)\to A$. It is clear that the degree of the 
minimal relation set of $A$ is no more than $N$ if and only 
if $\Phi_N(A)=A$. Namely,
\begin{equation}
\label{E4.2.1}\tag{E4.2.1}
N\geq \beta_2(A) \quad \Longleftrightarrow \quad \Phi_N(A)=A.
\end{equation}

\begin{lemma}
\label{xxlem4.3} 
Let $f: A \rightarrow B$ be as in Notation  \ref{xxnot4.1}.
If $N\geq \max\{\beta(A),\beta_2(A)\}$, then there is a 
unique lifting of the map $f$ to a map $f': A\to \Phi_N(B)$.
\end{lemma}

\begin{proof}  
By definition, we have $B_{\leq N}=\Phi_{N}(B)_{\leq N}$,
so we will identify these two graded spaces. If $f'$ 
exists, then since $N \geq \beta(A)$, we have 
$f'\mid_{\mathcal{G}(A)}=f\mid_{\mathcal{G}(A)}$.
Therefore if $f'$ exists, it is unique. 

Next we prove the existence of $f'$. Let $\Pi$ be the canonical 
map $\Bbbk\langle \mathcal{G}(A)\rangle \to A 
(=\Bbbk\langle \mathcal{G}(A)\rangle /(I(A)))$. 
Since $\Bbbk\langle \mathcal{G}(A)\rangle$ is a free algebra, 
we can lift the map $f: A\to B$ to 
$g: \Bbbk\langle \mathcal{G}(A)\rangle \to B$ by defining 
$$g (x):=f(x)=(f\circ \Pi)(x), \quad \text{for all }x 
\in \mathcal{G}(A).$$ 
It follows that $g=f\circ \Pi$. Since $\Bbbk\langle \mathcal{G}(A)\rangle$ 
is a free algebra, there is an algebra map $f'': 
\Bbbk\langle \mathcal{G}(A)\rangle \to \Phi_N(B)$ defined by setting 
$f''(x):=f(x) $ for all $x\in \mathcal{G}(A)$. Since $B_{\leq N}=
\Phi_{N}(B)_{\leq N}$, 
when restricted to $(\Bbbk\langle \mathcal{G}(A)\rangle)_{\leq N}$, 
$f''=\pi_{B} \circ f''=g$. Therefore $g=\pi_B \circ f''$.
It remains to show that $f''(I(A))=0$. By construction, 
$g(I(A))=0$. Then 
$$f''((I(A))_{\leq N})=\pi_B\circ f''((I(A))_{\leq N})=
g((I(A))_{\leq  N})=0$$
as $\pi_B$ is the identity when restricted to elements 
of degree no more than $N$. Since $I(A)$ is generated by 
$(I(A))_{\leq N}$, we obtain that $f''(I(A))=0$, as desired.
We now let $f'$ be the map $A \to \Phi_N(B)$ induced by $f''$.
\end{proof}

\begin{lemma}
\label{xxlem4.4} 
Let $f: A \rightarrow B$ be as in Notation  \ref{xxnot4.1}. If 
$$N\geq \max\{ 2 t^A_0(_A B), t^A_0(_A B)+\beta(A), 
\beta_2(A)\}$$
then $\Phi_N(B)$ is generated by $\{v_i\}_{i\in S_1}$ 
as a left $A$-module.
\end{lemma}

\begin{proof} 
We identify $a\in A$ with $f(a)\in B$ 
and $f'(a)\in \Phi_N(B)$ when there is no confusion.
Next we express some products  of elements in $B$.
For each $i\in S_1$ and $j\in S_2$ we write
\begin{equation}
\label{E4.4.1}\tag{E4.4.1}
v_i x_j=\sum_{k} y_{ijk} v_k,
\end{equation}
for some $y_{ijk}\in A$; and for $i, j\in S_1$,
\begin{equation}
\label{E4.4.2}\tag{E4.4.2}
v_i v_j=\sum_{k} z_{ijk} v_k,
\end{equation}
for some $z_{ijk}\in A$. The above two equations are 
in degrees no more than 
$$\max\{ 2 t^A_0(_A B), t^A_0(_A B)+\beta(A)\} \leq N.$$ 
Since we can identify $B_{\leq N}$ with 
$\Phi_{N}(B)_{\leq N}$, equations \eqref{E4.4.1}--\eqref{E4.4.2} 
hold in $\Phi_{N}(B)$. 

We claim that $\Phi_{N}(B)=\sum_{i\in S_1} A v_i$.
By the proof of Lemma \ref{xxlem4.2}, $B$, whence 
$\Phi_{N}(B)$, is generated by a subset
$W \subset \{v_i\}_{i\in S_1}\cup \{x_j\}_{j\in S_2}$, 
and it is enough to show that 
$$\left(\sum_{i\in S_1} A v_i\right) W\subseteq 
\left(\sum_{i\in S_1} A v_i\right).$$
But this statement follows from equations 
\eqref{E4.4.1}--\eqref{E4.4.2} viewed in $\Phi_N(B)$.
\end{proof}

Next is a version of \cite[Proposition 2.1(2,3)]{Sy}.

\begin{proposition}
\label{xxpro4.5} 
Let $f: A \rightarrow B$ be as in Notation \ref{xxnot4.1}.
Then 
$$\beta_2(B)\leq \max\left\{2 t^A_0(_A B),
t^A_0(_A B)+\beta(A),\beta_2(A), 
t^A_1(_A B)\right\}.$$
\end{proposition}

\begin{proof} Let 
$$N=\max\left\{2 t^A_0(_A B),
t^A_0(_A B)+\beta(A),\beta_2(A), 
t^A_1(_A B)\right\}.$$
By equation \eqref{E4.2.1}, it is enough to show that 
$\Phi_N(B)=B$. By Lemma \ref{xxlem4.3}, $f$ lifts to 
an algebra map $f'$ from $A$ to $\Phi_{N}(B)$. Now we 
have the following commutative diagram of left $A$-modules 
with exact rows, where the vertical arrows can be filled 
in since the $P_i$ are projective:
$$\begin{CD}
\cdots @>>> P_1 @>>> P_0 @>>> B @>>> 0\\
@. @VVV @VVV @VVV \\
0@>>> K @>>> \Phi_N(B) @>>>  B @>>> 0.
\end{CD}
$$
The map from $B$ to $B$ is the identity. By the 
definition of $\Phi_N(B)$, $K_{\leq N}=0$. When we 
take a minimal resolution of the top row of the above 
diagram, $P_1$ is generated in degree at most $t^A_1(_A B)$, 
which is $\leq N$. So the map from $P_1$ to $K$ is zero, 
thus the composition map $P_1\to \Phi_N(B)$ is zero.
It follows that the bottom row is split as a sequence
of left $A$-modules. By Lemma \ref{xxlem4.4}, 
$\Phi_N(B)$ is generated as a left $A$-module in degree 
at most $N$. Hence $K$ is generated as a left $A$-module 
in degree at most $N$. This implies that $K=0$. 
\end{proof}

Since the CM regularity of an AS regular algebra is easy to compute
(see Example \ref{xxex2.10}(3)), the following corollary is 
useful in bounding $\beta$ or $\beta_2$. 

\begin{corollary}
\label{xxcor4.6} 
Assume Hypothesis \ref{xxhyp0.3}. Suppose there is a graded 
algebra map $S\to T^H$ where $S$ is a noetherian AS regular 
algebra such that $T^H$ is finitely generated over $S$ on both
sides. Let $\delta(T/S)=\CMreg(T)-\CMreg(S)$. Then
\begin{enumerate}
\item[(1)] $\beta(T^H)\leq \max\{\beta(S), \delta(T/S)\}$, and
\item[(2)] $\beta_2(T^H)\leq \max\left\{2 \delta(T/S),
\delta(T/S)+\beta(S),\beta_2(S)\right\}$.
\end{enumerate}
\end{corollary}

\begin{proof}
Let $A=S$ and $B=T^H$. 
Then by Example \ref{xxex2.12}(1),  Theorem 
\ref{xxthm2.13}(3), and Lemma \ref{xxlem2.15}(1),  
we have, for all $i\geq 0$,
\begin{align}
\label{E4.6.1}\tag{E4.6.1}
t^A_i(_A B)&\leq  \Torreg(_A B)+i =\CMreg(B)-\CMreg(A)+i\\
\notag &\leq \CMreg(T)-\CMreg(S)+i=\delta(T/S)+i.
\end{align}

(1) The assertion follows from Lemma \ref{xxlem4.2} and the inequality
\eqref{E4.6.1}.

(2) The assertion follows from Proposition \ref{xxpro4.5},
the inequality \eqref{E4.6.1}, and the fact that $\beta(S)\geq 1$.
\end{proof}

In the remainder of this section, we prove a noncommutative 
version of \cite[Theorem 1]{De}, which provides a bound 
on the degrees of higher syzygies of $T^H$. Recall from 
Definition \ref{xxdef2.1} that the {\it $a$-invariant} of a 
graded module $M$, denoted by $a(M)$, is defined to be the 
$t$-degree of the Hilbert series $h_M(t)$, viewed as a rational 
function. 
Note that if $A  = \Bbbk[x_1, \dots, x_r]$ is
the commutative polynomial ring such that $\deg x_i = d_i$ and
$d_i \geq d_{i+1}$, then for every $1 \leq k \leq r$, we have
\begin{equation}
\label{E4.6.2}\tag{E4.6.2}
t^A_{k}(\Bbbk) = d_1 + \cdots + d_k.
\end{equation}
Hence, the next theorem generalizes \cite[Theorem 6]{De}.

\begin{theorem}
\label{xxthm4.7}
Let $A$ be a noetherian connected graded algebra with 
balanced dualizing complex, and let $M$ be a finitely 
generated graded left $A$-module that is $s$-Cohen--Macaulay.
\begin{enumerate}
\item[(1)]
For each $i$, 
$$t^A_i(_AM)
\leq \CMreg(M) -s+t^A_{i+s}(\Bbbk).$$
\item[(2)]
Assume Hypothesis \ref{xxhyp2.7} for $A$. Then, for each $i$, 
$$t^A_i(_AM)
\leq a(M)+t^A_{i+s}(\Bbbk).$$
\end{enumerate}
\end{theorem}

\begin{proof} The proof given here is different from the 
proof of \cite[Theorem 6]{De}.

(1) Let
$$ \cdots \to F_{m}\to \cdots \to F_1 \to F_0\to \Bbbk\to 0$$
be a minimal free resolution of the right $A$-module 
$\Bbbk$. Then $F_m=\coprod_j A(-\sigma_{m,j})$ with
$$\sigma_{m,j}\leq t^A_{m}(\Bbbk).$$
Taking graded $\Bbbk$-linear duals, we obtain a 
minimal injective resolution of the left trivial module
$$0\to \Bbbk\to F'_0\to F'_1 \to \cdots \to F'_{m}\to \cdots$$
where $F'_m= \coprod_j A'(\sigma_{m,j})$. 

Let $X$ be a graded left $A$-module that is bounded 
above. Then $\Ext^m_A(X, \Bbbk)$ is a subquotient of 
$$\Hom_{A}(X, F'_m) = \bigoplus_j \Hom_A(X, A'(\sigma_{m,j}))
= \bigoplus_j X'(\sigma_{m,j}).$$
Hence
\begin{equation}
\label{E4.7.1}\tag{E4.7.1}
\ged(\Ext^m_A(X, \Bbbk)) \geq \ged (X') -\max_j\{ \sigma_{mj}\}
=\ged (X') -t^A_{m}(\Bbbk).
\end{equation}

Since $M$ is $s$-Cohen--Macaulay, $H^i_{\fm}(M)=0$ for 
all $i\neq s$. Let $X=H^s_{\fm}(M)$. By 
Example \ref{xxex2.10}(4), $\deg(X)=\CMreg(M)-s$, or
equivalently,
\begin{equation}
\notag\label{E4.7.2}\tag{E4.7.2}
\ged(X')=-\CMreg(M)+s.
\end{equation}

By \cite[Proposition 1.1]{Jo3}, $\RHom_A(M,\Bbbk)) 
\cong \RHom_{A}({\text{R}}\Gamma_{\fm}(M),\Bbbk)$. Hence
$$\begin{aligned}
\Ext^i_A(M,\Bbbk)
&=H^i(\RHom_A(M,\Bbbk)) 
\cong H^i(\RHom_{A}({\text{R}}\Gamma_{\fm}(M),
\Bbbk))\\
&\cong H^i(\RHom_A(X[-s], \Bbbk)\cong H^{i+s}(\RHom_A(X,\Bbbk))\\
&=\Ext^{i+s}_{A}(X,\Bbbk)
\end{aligned}
$$
which implies that
$$\begin{aligned}
\ged (\Ext^i_A(M,\Bbbk)) 
&\;\;\;\;=\;\;\;\; \ged (\Ext^{i+s}(X, \Bbbk))\\
&\overset{\eqref{E4.7.1}}{\geq}
\ged (X') -t^A_{i+s}(\Bbbk) \\
&\overset{\eqref{E4.7.2}}{=}
-\CMreg(M)+s-t^A_{i+s}(\Bbbk).
\end{aligned}
$$
This is equivalent to
$$\deg ((\Ext^i(M,\Bbbk))')\leq \CMreg(M)-s+t^A_{i+s}(\Bbbk).$$
By \cite[Remark 4.5]{DW}, 
$(\Ext^i(M,\Bbbk))'\cong \Tor^A_i(\Bbbk, M)$. Thus
$$t^A_i(_AM):=\deg (\Tor^A_i(\Bbbk, M))\leq 
\CMreg(M)-s+t^A_{i+s}(\Bbbk).$$

(2) By definition and Hypothesis \ref{xxhyp2.7}, we have
$$\begin{aligned}
a(M)&\overset{{\text{by def.}}}{=}
\deg_t h_{M}(t) \overset{{\text{Hyp. \ref{xxhyp2.7}}}}{=}
\deg_t h_{H^s_{\fm}(M)}(t) \\
&\overset{\eqref{E2.2.1}}{=}
\deg H^s_{\fm}(M) 
\overset{\eqref{E2.10.2}}{=}
\CMreg(M)-s. 
\end{aligned}
$$ 
Now the assertion follows from part (1).
\end{proof}

As a corollary, we obtain a noncommutative version of 
\cite[Theorem 1]{De}. 
Let $T = \Bbbk[x_1, \dots, x_n]$ and suppose we have
an action of $H = \Bbbk G$ where $G$ is a finite group.
Choose a minimal set of homogeneous generators $y_1, \dots, y_r$ of the invariant ring $T^H$ where $\deg(y_i) = d_i$ and $d_i \geq d_{i+1}$ for all $i$.
Then there is a graded map from $S$, the polynomial ring in $r$ variables (with degree given by the $d_i$'s) onto $T^H$.
By \eqref{E4.6.2}, and the fact that each $d_i \leq \beta(T^H)$, we have
that $t^S_{i+n}(\Bbbk) \leq (i+n) \beta(T^H)$ for all $i$,
and hence the corollary below recovers Derksen's result.

\begin{corollary}
\label{xxcor4.8} 
Assume Hypothesis~\ref{xxhyp0.3} and let $R=T^H$. 
Suppose that $S$ is a noetherian AS regular algebra and 
there exists a graded algebra homomorphism $S \to R$ such 
that $R$ is finitely generated over $S$ on both sides. Let 
$n$ be the global dimension of $T$. Then 
$$t^S_i(_S R) \leq \CMreg(T) -n +t^S_{i+n}(\Bbbk) \leq
t^S_{i+n}(\Bbbk)-n$$
for all $i\geq 0$.
\end{corollary}

\begin{proof} 
By Lemma \ref{xxlem2.15}(1) and Example 
\ref{xxex2.10}(3), 
$$\CMreg(R)=\CMreg(T^H)\leq \CMreg(T)\leq 0.$$ 
By Lemma \ref{xxlem2.6}(2), $R$ is $n$-Cohen--Macaulay.
The assertion follows from Theorem \ref{xxthm4.7}(1).
\end{proof}

We remark that Corollary \ref{xxcor4.8} is stronger than 
(though almost equivalent to) Lemma \ref{xxlem2.15}(1). 

\section{Further bounds on the degree of $\Tor_i^A(M,\Bbbk)$}
\label{xxsec5}
In this section we continue to prove bounds on the degrees 
of the higher syzygies of $T^H$, obtaining results that are 
similar to the main results in \cite{CS}. The results in this 
section require the existence of a graded algebra surjection 
from an AS regular algebra $S$ onto $T^H$. We begin with some 
general lemmas.

\begin{lemma}
\label{xxlem5.1}
Suppose that $f: A\to B$ is a graded algebra map 
between two connected graded algebras $A$ and $B$. 
\begin{enumerate}
\item[(1)] \cite[Theorem 10.59]{Ro}
There is a spectral sequence, called the 
change of rings spectral sequence,
$$E^2_{p,q}:=\Tor^B_{p} (\Bbbk_{B}, \Tor^A_q(B, {_A \Bbbk})) 
\Longrightarrow \Tor^A_{p+q} (\Bbbk_{A}, {_A \Bbbk}).$$
The five-term exact sequence associated to the spectral 
sequence is
$$\begin{aligned}
\Tor^A_2(\Bbbk,\Bbbk) &\to \Tor^B_2(\Bbbk, B\otimes_A \Bbbk)
\to \Bbbk\otimes_B \Tor^A_1(B,\Bbbk)\\ 
&\to \Tor^A_1(\Bbbk,\Bbbk)
\to \Tor^B_1(\Bbbk, B\otimes_A \Bbbk)\to 0.
\end{aligned}$$
\item[(2)]
Suppose that $f$ is surjective and that
$A$ and $B$ have the same minimal generating set, that is,
$\{x_j\}$ is a minimal generating set for $A$ and $\{f(x_j)\}$ 
is a minimal generating set for $B$. Then we have an exact 
sequence of graded vector spaces
$$\Tor^A_2(\Bbbk, \Bbbk) \to \Tor^B_2(\Bbbk, \Bbbk)
\to \Bbbk\otimes_B \Tor^A_1(B, \Bbbk)\to 0.$$
\item[(3)]
Retain the hypotheses in part {\rm{(2)}}.
If $\Bbbk\otimes_B \Tor^A_1(B,\Bbbk)\cong \Tor^A_1(B, \Bbbk)$, 
then we have the exact sequence.
\begin{equation}
\label{E5.1.1}\tag{E5.1.1}
\Tor^A_2(\Bbbk, \Bbbk) \to \Tor^B_2(\Bbbk, \Bbbk)
\to \Tor^A_1(B, \Bbbk)\to 0
\end{equation}
\end{enumerate}
\end{lemma}

\begin{proof} 
(1) This is a special case of \cite[Theorem 10.59]{Ro}.
The five term exact sequence is given immediately after 
\cite[Theorem 10.59]{Ro}.

(2) If $f$ is surjective, then $B\otimes_A \Bbbk=\Bbbk$. 
Hence we have an exact sequence
$$\begin{aligned}
\Tor^A_2(\Bbbk,\Bbbk) &\to \Tor^B_2(\Bbbk, \Bbbk)
\to \Bbbk\otimes_B \Tor^A_1(B,\Bbbk)\\ 
&\to \Tor^A_1(\Bbbk,\Bbbk)
\to \Tor^B_1(\Bbbk, \Bbbk)\to 0.
\end{aligned}$$
Since $A$ and $B$ have the same minimal generating set,
$\Tor^A_1(\Bbbk,\Bbbk) \cong \Tor^B_1(\Bbbk, \Bbbk)$. 
So the assertion follows.

(3) This follows immediately from part (2).
\end{proof}

Part of the $E^2$ page of the spectral sequence in 
Lemma \ref{xxlem5.1}(1) looks like
$$\begin{matrix}
\Bbbk \otimes_B \Tor^A_3(B, \Bbbk) &
\Tor^B_{1} (\Bbbk, \Tor^A_3(B, \Bbbk)) &
\Tor^B_{2} (\Bbbk, \Tor^A_3(B, \Bbbk)) &  
\Tor^B_{3} (\Bbbk, \Tor^A_3(B, \Bbbk))  \\
\Bbbk \otimes_B \Tor^A_2(B, \Bbbk) &
\Tor^B_{1} (\Bbbk, \Tor^A_2(B, \Bbbk)) &
\Tor^B_{2} (\Bbbk, \Tor^A_2(B, \Bbbk)) &  
\Tor^B_{3} (\Bbbk, \Tor^A_2(B, \Bbbk)) \\
\Bbbk \otimes_B \Tor^A_1(B, \Bbbk) &
\Tor^B_{1} (\Bbbk, \Tor^A_1(B, \Bbbk)) &
\Tor^B_{2} (\Bbbk, \Tor^A_1(B, \Bbbk)) &  
\Tor^B_{3} (\Bbbk, \Tor^A_1(B, \Bbbk))  \\
\Bbbk \otimes_B (B\otimes_A \Bbbk) &
\Tor^B_{1} (\Bbbk, B\otimes_A \Bbbk) &
\Tor^B_{2} (\Bbbk, B\otimes_A \Bbbk) &  
\Tor^B_{3} (\Bbbk, B\otimes_A \Bbbk)  \\
\end{matrix}
$$

The differential on the $E^r$-page has degree $(-r, r-1)$,
namely, $d^r: E^r_{p,q}\to E^r_{p-r,q+r-1}$. For example,
the differential on the $E^2$-page is
$$d^2: \Tor^B_{p} (\Bbbk_{B}, \Tor^A_q(B, {_A \Bbbk}))
\to \Tor^B_{p-2} (\Bbbk_{B}, \Tor^A_{q+1}(B, {_A \Bbbk}))$$
for all $(p,q)$.

From now on suppose $B_A$ is finitely generated and $A$ is 
right noetherian. Then $\Tor^A_q(B, \Bbbk)$ is 
finite-dimensional. Filtering $\Tor^A_q(B,\Bbbk)$
by degree, we see that
\begin{equation}
\label{E5.1.2}\tag{E5.1.2}
\begin{aligned}
\deg \Tor^B_{p} (\Bbbk, \Tor^A_q(B, \Bbbk))
& \leq \deg \Tor^A_q(B, \Bbbk)+\deg \Tor^B_{p} (\Bbbk, \Bbbk)\\
&= t^A_q(B_A) +t^B_p(\Bbbk).
\end{aligned}
\end{equation}

The degree of an entry on the $E^2$ page is bounded by the 
maximum of the degree of $H_i(\text{Tot})$ corresponding to its 
diagonal and the degrees of the $E^2$ entries that are linked
to it by a differential on some page. Applying this to the 
bottom row yields
\begin{equation}
\label{E5.1.3}\tag{E5.1.3}
\deg \Tor^B_{i} (\Bbbk, B\otimes_A \Bbbk)
\leq \max\{
\{t^B_j(\Bbbk)+t^A_{i-j-1}(B_A)\}_{0\leq j\leq i-2},t^A_i(\Bbbk)\}.
\end{equation}
Considering the first column, we obtain
\begin{equation}
\label{E5.1.4}\tag{E5.1.4}
\deg \Bbbk \otimes_B \Tor^A_{i} (B, \Bbbk)
\leq \max\{
\{t^A_j(B_A)+t^B_{i-j+1}(\Bbbk)\}_{0\leq j\leq i-1},t^A_i(\Bbbk)\}.
\end{equation}

Similar to Definition \ref{xxdef1.1}(3), for each $i$, let 
$J_i\subseteq B$ be the annihilator ideal of the finite-dimensional 
left $B$-module $\Tor^A_i(B,\Bbbk)$. Let $J_{\leq i}$ denote 
$\bigcap_{j\leq i} J_i$; when $i$ is clear we will use $J$ for 
$J_{\leq i}$. Notice that $\Tor^A_i(B,\Bbbk)$ is naturally a 
graded left $B/J$-module and is generated as such in degrees 
at most $\deg (\Bbbk \otimes_B \Tor^A_i(B,\Bbbk))$.  Thus
\begin{equation}
\label{E5.1.5}\tag{E5.1.5}
t^A_i(B_A):=\deg \Tor^A_i(B,\Bbbk) \leq 
\deg (\Bbbk \otimes_B \Tor^A_i(B,\Bbbk))+\deg B/J_{\leq i}.
\end{equation}

For each non-negative integer $i$, let $D_i$ be a positive 
number which is greater than or equal to 
$$\max\left\{ \deg B/J_{\leq i}+ t^B_2(\Bbbk), 
\left\{\frac{t^B_{j+2}(\Bbbk)-t^B_2(\Bbbk)}{j}\right\}_{1 \leq j\leq i-1},
\quad
\left\{\frac{t^B_j(\Bbbk)}{j}\right\}_{1 \leq j\leq i}\right\}.$$

For every $j\leq i$, set
\begin{equation}
\label{E5.1.6}\tag{E5.1.6}
U^i_j(f):=\max_{i_s>0,\sum_{s} i_s=j}\left\{\sum_{s} 
\left(t^B_{i_s+1}(\Bbbk)+D_i-t^B_2(\Bbbk) \right)\right\}
\end{equation}
for $j>0$ and define $U^i_j(f):=-\infty$ for $j\leq 0$. For 
example,
$U^i_1(f)=D_i$ and $U^i_2(f) = 
\max\left\{2 D_i, t_{3}^B(\Bbbk) -t^B_2(\Bbbk)+ D_i \right\}$.
By definition, for $j+k\leq i$,
\begin{equation}
\label{E5.1.7}\tag{E5.1.7}
U^i_{j+k}(f)\geq U^i_j(f) + t^B_{k+1}(\Bbbk)+D_i - t^B_2(\Bbbk).
\end{equation}

\begin{notation}For the remainder of the section, we fix notation as in Lemma~\ref{xxlem5.1}. In particular, if $f:A \to B$ is a graded algebra map between connected graded domains, and $i$ is a positive integer, then we let $D_i$ be the positive integer defined after \eqref{E5.1.5}. Further, for every $j \leq i$, we let $U^i_j(f)$ be the value defined in \eqref{E5.1.6}.
\end{notation}

\begin{proposition}
\label{xxpro5.2}
Fix a positive integer $i$ and retain the above notation. For each 
$j \leq i$, we have
$$t^A_j(B_A)\leq \max\left\{U^i_j(f), 
\{t^A_k(\Bbbk)+(j-k) D_i\}_{0\leq k\leq j}\right\}
+D_i-t^B_2(\Bbbk).$$
\end{proposition}

\begin{proof}
We prove the assertion by induction on $j$. First let $j = 0$. 
By the definition of $J_{\leq 0}=\ann_B (B \otimes_A \Bbbk)$, 
we have $J_{\leq 0} \cdot(B \otimes_A \Bbbk) = 
J_{\leq 0} \cdot( B/ BA_+) = 0$ so $J_{\leq 0}\subseteq BA_+$. 
Then
\[ t^A_0(B_A)=\deg B/BA_{+}\leq \deg B/J_{\leq 0}
\leq \deg B/J_{\leq i}\leq D_i - t^B_2(\Bbbk)
\]
which is the assertion when $j = 0$.

For the inductive step, we assume that $i > 0$. Fix $j \leq i$.
Then for $0\leq k\leq j-1$,
$$\begin{aligned}
t^A_k(B_A)&+t^B_{j-k+1}(\Bbbk)\\
&\overset{\text{ind. hyp.}}{\leq}
\max\left\{U^i_k(f), 
\{t^A_{\ell}(\Bbbk)+(k-{\ell})D_i\}_{0\leq \ell \leq k}\right\}
+D_i - t^B_2(\Bbbk) \\
&\qquad \qquad +t^B_{j-k+1}(\Bbbk)
\\
&\;\;\;\; \leq  \;\;\;\;
\max\left\{U^i_k(f) + D_i - t^B_2(\Bbbk)+t^B_{j-k+1}(\Bbbk),\right. \\
& \qquad \qquad
\left. \{t^A_{\ell}(\Bbbk) + (k-\ell)D_i + D_i - t^B_2(\Bbbk)
+t^B_{j-k+1}(\Bbbk)\}_{0\leq \ell\leq k}\right\}\\
&\overset{\eqref{E5.1.7}}{\leq} 
\max\{U^i_j(f), \\
& \qquad  
\{t^A_{\ell}(\Bbbk)+(j-\ell)D_i + (k+1-j)D_i - t^B_2(\Bbbk)
+t^B_{j-k+1}(\Bbbk)\}_{0\leq \ell\leq k}\}\\
&\; \overset{\text{def.} D_i}{\leq}  \;
\max\left\{U^i_j(f), \{t^A_{\ell}(\Bbbk)+(j-\ell)D_i\}_{0\leq \ell \leq k}\right\}\\
&\;\;\;\; \leq \;\;\;\;
\max\{U^i_j(f), \{t^A_{\ell} (\Bbbk)+(j-\ell)D_i\}_{0\leq \ell\leq j-1}\}.
\end{aligned}
$$
We use the above inequality to see that
$$\begin{aligned}
t^A_j(B_A)
& \overset{\eqref{E5.1.5}}{\leq} 
\deg (\Bbbk \otimes_B \Tor^A_j(B,\Bbbk))+\deg B/J_{\leq j} \\
&\;\overset{\text{def. } D_i}{\leq}\;
\deg (\Bbbk \otimes_B \Tor^A_j(B,\Bbbk)) + D_i - t^B_2(\Bbbk)\\
& \overset{\eqref{E5.1.4}}{\leq }
\max\{
\{t^A_k(B_A)+t^B_{j-k+1}(\Bbbk)\}_{0\leq k \leq j-1},
t^A_j(\Bbbk)\} +D_i - t^B_2(\Bbbk)\\
&\;\;\;\; \leq \;\;\;\;
\max\{U^i_j(f), \{t^A_{\ell}(\Bbbk)+(j-{\ell})D_i\}_{0\leq \ell \leq j-1},
t^A_j(\Bbbk)\} +D_i-t^B_2(\Bbbk)\\
&\;\;\;\; = \;\;\;\; 
\max\{U^i_j(f), \{t^A_{\ell}(\Bbbk)+(j-{\ell})D_i\}_{0\leq {\ell} \leq j}\} 
+D_i - t^B_2(\Bbbk),
\end{aligned}
$$
completing  the proof.
\end{proof}

\begin{lemma}
\label{xxlem5.3}
Retain the hypotheses of Proposition \ref{xxpro5.2}. 
Then, for every $j \leq i$, $U^i_j(f)\leq j D_i$.
\end{lemma}

\begin{proof} By definition, 
\[D_i\geq \left\{\frac{t^B_{j+1}(\Bbbk)-
t^B_2(\Bbbk)}{j-1}\right\}_{2 \leq j\leq i}\]
which is equivalent to
\[t^B_{j+1}(\Bbbk)+D_i-t^B_2(\Bbbk)\leq j D_i
\]
for all $1\leq j\leq i$. Now the assertion 
follows easily from the definition of $U^i_j(f)$ 
\eqref{E5.1.6}.
\end{proof}

\begin{lemma}
\label{xxlem5.4}
Retain the hypotheses of Proposition \ref{xxpro5.2} 
and assume that, for all $0< k \leq i$,
\begin{equation}
\label{E5.4.1}\tag{E5.4.1}
D_i \geq \frac{t^A_k(\Bbbk)}{k}.
\end{equation}
Then for all $1 \leq j \le i$, 
\[t^A_j(B_A)\leq (j+1) D_i - t^B_2(\Bbbk).\]
\end{lemma}

\begin{proof} 
Under the hypothesis, we have that for all 
$1 \leq  j, k \leq i$,
\[t^A_k(\Bbbk)+(j-k)D_i \leq j D_i.
\]
By Lemma~\ref{xxlem5.3}, we also have that $U^i_j(f) \leq j D_i$.
Hence, by Proposition~\ref{xxpro5.2}, 
\begin{align*}t^A_j (B_A) &\leq \max\left\{U^i_j(f), 
\{t^A_k(\Bbbk)+(j-k) D_i\}_{0\leq k\leq j}\right\}
+D_i-t^B_2(\Bbbk) \\ 
&\leq j D_i + D_i - t^B_2(\Bbbk) = (j+1)D_i - t^B_2(\Bbbk),
\end{align*}
as desired.
\end{proof}

Now we have an immediate consequence.

\begin{corollary}
\label{xxcor5.5} 
Retain the hypotheses of Proposition \ref{xxpro5.2}. 
Suppose that $D_i$ is a number larger than or equal to
$$\max\left\{ \deg B/J_{\leq i}+ t^B_2(\Bbbk), 
\left\{\frac{t^B_{j+2}(\Bbbk)-t^B_2(\Bbbk)}{j}\right\}_{1 \leq j\leq i-1},
\left\{\frac{t^A_j(\Bbbk)}{j}\right\}_{1 \leq j\leq i},
\left\{\frac{t^B_j(\Bbbk)}{j}\right\}_{1 \leq j\leq i}
\right\}.$$
Then for all $1 \leq j \leq i$,
$t^A_j(B_A)\leq (j+1)D_i - t^B_2(\Bbbk)$.
\end{corollary}

This recovers the result in the commutative case \cite[Corollary 5.2]{CS}.
To see this note that in the setting of \cite[Corollary 5.2]{CS}, 
$B=\Bbbk[x_1,\cdots,x_n]$ with $\deg x_i=1$ for all $i$ and 
$A=\Bbbk[y_1,\cdots,y_n]$ with $\deg y_i=d_i$ with
$\{d_1,\cdots,d_n\}$ non-increasing. Then $t^B_j(\Bbbk)=\begin{cases}
j& 0\leq j\leq n\\ 0 & {\text{otherwise}}\end{cases}$, and 
$t^A_j(\Bbbk)/j\leq d_1 \leq \deg B/BA_{\geq 1}=\deg B/J_{\geq i}$
for all $j\leq i$. Thus we can take $D_i=\deg B/BA_{\geq 1}+2$
(which is independent of $i$).
By Corollary \ref{xxcor5.5}, we have 
$$t^A_j(B)\leq (j+1) (\deg B/BA_{\geq 1}+2)-2$$
which is the second statement of \cite[Corollary 5.2]{CS}.
The first statement of \cite[Corollary 5.2]{CS} follows from
Proposition \ref{xxpro5.2} (details are omitted). 

Now suppose $H$ is a semisimple Hopf algebra acting on $B$ 
homogeneously and let $C=B^H$. Suppose that $f: A\to C$ is a 
surjective map of graded algebras and consider it as a graded 
algebra map $A\to B$. 
Recall that  $J_i\subseteq B$ 
denotes the annihilator ideal of the finite-dimensional left 
$B$-module $\Tor^A_i(B,\Bbbk)$. 

\begin{theorem}
\label{xxthm5.6}
Let $(T,H)$ be as in Hypothesis \ref{xxhyp0.3} and assume 
that $T$ is Koszul. Let $J_{\infty}=\cap_{j\geq 0} J_j$. Let 
$S$ be a noetherian AS regular algebra that maps onto $R:=T^H$ 
surjectively such that $t^S_j(\Bbbk)\leq j (\deg T/J_{\infty}+2)$ 
for all $j\geq 0$. Then 
$$t^S_i(R_S)\leq i(\deg T/J_{\infty} + 2)+\deg T/J_{\infty}$$
for all $i\geq 0$.
\end{theorem}

\begin{proof} 
Since $T$ is Koszul, $t^T_j(\Bbbk)=j$ for all $0\leq j\leq \gldim T$. 
In particular, $t^T_2(\Bbbk)=2$. Under the hypotheses of this theorem, 
one can check that $\deg T/J_{\infty}+2$ is at least equal to each 
term in the $\max$-expression in Corollary \ref{xxcor5.5} 
(letting $(A,B) = (S,T)$). Therefore, if we take 
$D=\deg T/J_{\infty}+2$ and apply Corollary \ref{xxcor5.5}, we obtain
that 
$$t^S_i(T_S)\leq (i+1) D - t^T_2(\Bbbk)= i D+D-2.$$ 
Since $R$ is a direct summand of $T$ as a right $S$-module,
the assertion follows.
\end{proof}

This is a noncommutative version of \cite[Theorem 1.2 (part 1)]{CS}.
To see this we let $T$ be the polynomial ring
$\Bbbk[x_1,\cdots,x_n]$ with $\deg x_i=1$ (that is $B$ in 
\cite[Theorem 1.2]{CS}). Let $S$ be another polynomial ring 
mapping onto $R:=T^H$ (where $H=\Bbbk G$ for some finite
group in the setting of \cite[Theorem 1.2 (part 1)]{CS}). 
In the commutative case 
$$D:=\deg T/J_{\infty} + 2=\deg T/TR_{\geq 1}+2=\tau_H(T)+1
\leq |G|+1$$
where the last $\leq$ follows from Fogarty's result 
\cite{Fo}, or equivalently, Proposition \ref{xxpro3.11}(1) by 
taking $d=1$ and $m=0$. By Theorem \ref{xxthm5.6}, 
$$\begin{aligned}
t^S_i(_SR)\leq i D+D-2&=(i+1)D -2=(i+1)(\tau_{H}(T)+1)-2\\
&=(i+1)\tau_{H}(T)+i-1\leq (i+1)|G|+i-1
\end{aligned}
$$
which recovers exactly \cite[Theorem 1.2 (part 1)]{CS}).

\begin{remark}
\label{xxrem5.7}
When $T^H$ is noncommutative, the hypothesis in Theorem 
\ref{xxthm5.6} does not hold automatically. While there are 
cases where it is known that such an AS regular algebra $S$ 
exists (see e.g. \cite{KKZ5}, \cite{CKWZ2}), it is unknown 
if this holds in general. For a general connected graded 
algebra $A$, we can make the following comments.
\begin{enumerate}
\item[(1)]
If $A$ is finite-dimensional, then there is a noetherian AS 
regular algebra $S$ and a surjective algebra map $f: S\to A$
\cite[p. 34]{PZ}.
\item[(2)]
Let $A$ be the noetherian connected graded domain given in 
\cite[Theorem 2.3]{SZ}. Then $A$ has GK dimension 2 and 
does not satisfy the $\chi$-condition. For each integer 
$d\geq 2$, let $B$ be the polynomial extension 
$A[x_1,\cdots,x_{d-2}]$. Then $B$ is a noetherian connected 
graded domain of GK-dimension $d$ that does not satisfy the 
$\chi$-condition \cite[Theorem 8.3]{AZ}. We claim that there 
is no surjective homomorphism from a noetherian AS regular 
algebra $S$ to $B$ (nor a graded algebra homomorphism 
from $S$ to $B$ such that $B$ is finitely generated over $S$
on both sides). Suppose to the contrary that there is an AS 
regular algebra $S$ and a surjective homomorphism from $S$ to 
$B$. By \cite[Theorem 8.1]{AZ}, $S$ satisfies the 
$\chi$-condition, and by \cite[Theorem 8.3]{AZ}, so does $B$. 
This yields a contradiction.
\item[(3)]
When $\GKdim A=1$, it is still unknown if there exists a
surjective homomorphism from a noetherian AS regular algebra 
$S$ to $A$.
\end{enumerate}
\end{remark}

Next we prove a noncommutative version of \cite[Theorem 1.3]{CS}.
One can recover \cite[Theorem 1.3]{CS} from Proposition 
\ref{xxpro5.8} by specializing to the commutative situation, but we
omit those details here. We return to the setting in Proposition 
\ref{xxpro5.2}. 

\begin{proposition}
\label{xxpro5.8} 
Let $f: A\to B$ be a graded algebra homomorphism of connected 
graded algebras and let $C=\im (f)$. Assume the following:
\begin{enumerate}
\item[(a)]
$B$ is generated in degree 1,
\item[(b)]
$C$ is a direct summand of $B$ as a right $A$-module,
\item[(c)]
$D_i$ is the number given in Corollary \ref{xxcor5.5}, and
\item[(d)] $\displaystyle D_i \geq \max\left\{\frac{t^A_j(\Bbbk)
+t^B_2(\Bbbk)}{j}\right\}_{2\leq j\leq i}.$
\end{enumerate}
Then $t^C_0(\Bbbk)=0$, $t^C_{1}(\Bbbk)\leq D_i - t^B_2(\Bbbk)+1$, 
and, for $2\leq j \leq i$, 
$$t^C_{j}(\Bbbk)\leq j D_i-t^B_2(\Bbbk).$$
\end{proposition}

\begin{proof} The assertion for $t^C_0(\Bbbk)$ is obvious.
The assertion for $t^C_1(\Bbbk)$ follows from Lemma \ref{xxlem3.2}(3).
Now let $ j \geq 2$ and let $F_j = jD_i - t^B_2(\Bbbk)$. It remains to show
that $t^C_j(\Bbbk) \leq F_j$. We proceed by induction on $j$.

Applying \eqref{E5.1.3} to the map $f: A\to C$, 
and using the fact that $C\otimes_A \Bbbk=\Bbbk$, we obtain
that
\begin{equation}
\label{E5.8.1}\tag{E5.8.1}
t^C_j(\Bbbk)\leq \max\left\{ \left\{t^C_k(\Bbbk)
+t^A_{j-k-1}(C_A)\right\}_{0\leq k\leq j-2}, 
\; \; t^A_j(\Bbbk)\right\}.
\end{equation}
Note that when $2 \leq j \leq i$, hypothesis (d) on $D_i$ is equivalent to
\[
t^A_j (\Bbbk) \leq j D_i -t^B_2(\Bbbk)=F_j.
\]
Hence, to show the main assertion it suffices to show that for all 
$0 \leq k \leq j- 2$,
\begin{equation}
\label{E5.8.2}\tag{E5.8.2}
t^C_k(\Bbbk)+t^A_{j - k - 1}(C_A) \leq F_j.
\end{equation}
Note that \eqref{E5.8.2}
holds for $k=0$ because 
\begin{equation}
\label{E5.8.3}\tag{E5.8.3}
t^C_0(\Bbbk)+t^A_{j-1}(C_A)=t^A_{j-1}(C_A)\leq t^A_{j-1}(B_A)
\leq F_j
\end{equation}
where the first inequality holds because $C$ is a direct summand
of $B$ as a right $A$-module, and the second inequality 
holds by Corollary \ref{xxcor5.5}. Further,
\eqref{E5.8.2} holds for $k=1$ because
$$t^C_1(\Bbbk)+t^A_{j-2}(C_A)
\leq D_i + F_{j-1}= F_{j}.$$
If $2\leq k \leq j-2$, we use the induction hypothesis and the fact
$t^A_{j-k-1}(C_A)\leq F_{j-k}$ (as explained in \eqref{E5.8.3}) to see that
$$
t^C_k(\Bbbk)+t^A_{j-k-1}(C_A)\leq F_k +F_{j-k}<F_j.$$
Therefore \eqref{E5.8.2} holds for all $k \leq j-2$,
as desired.
\end{proof}

\begin{theorem}
\label{xxthm5.9}
Retain the hypotheses of Theorem \ref{xxthm5.6}. Fix a positive 
integer $i$ and assume that
\[\deg T/J_{\infty} \geq \max\left\{\frac{t^S_j(\Bbbk)
+t^T_2(\Bbbk)}{j}\right\}_{2\leq j\leq i} - 2.\]
Then $t^R_0(\Bbbk)=0$, 
$t^R_{1}(\Bbbk)\leq \deg T/J_{\infty}+1$, and, for 
$2 \leq j \leq i$, 
$$t^R_{j}(\Bbbk)\leq j(\deg T/J_{\infty} + 2) - 2.$$
\end{theorem}

\begin{proof}
Note that since $T$ is Koszul, $t^T_2(\Bbbk)=2$. Letting 
$(A,B)=(S,T)$, observe that Hypothesis (d) in Proposition 
\ref{xxpro5.8} holds for $D_i = \deg T/J_{\infty} + 2$. 

Under these hypotheses, one can check that
$\deg T/J_{\infty} + 2$ is at least equal to each term in 
the $\max$-expression in Corollary \ref{xxcor5.5} and that
the hypotheses in Proposition \ref{xxpro5.8} hold.
Therefore the assertion follows from Proposition \ref{xxpro5.8}.
\end{proof}

The commutative result \cite[Theorem 1.3(part 1)]{CS} 
is covered by the above theorem. We now give a very 
special case of Theorem \ref{xxthm5.9}.

\begin{corollary}
\label{xxcor5.10}
Let $G$ be a finite group acting as graded automorphisms on 
$T:=\Bbbk_{-1}[x_1, \dots, x_n]$, and suppose that $\Bbbk$ is 
an infinite field and $\Bbbk G$ is semisimple. Assume that 
$R:=T^G$ is commutative.  Then
$t^{R}_0(\Bbbk)=0$, $t^R_1(\Bbbk)\leq 2|G|+n$, and 
$$t^{R}_i(\Bbbk)\leq i (2|G|+n+1)-2$$
for all $i\geq 2$.
\end{corollary}

\begin{proof} It is clear that $t^{R}_0(\Bbbk)=0$. By 
Corollary \ref{xxcor3.12}, $t^R_1(R)\leq 2|G|+n$. 

When $i\geq 2$, we let $S$ be a commutative 
polynomial ring generated by elements of degree 
$\leq \beta(R)$ which maps surjectively onto $R$.

Since $T:=\Bbbk_{-1}[x_1, \dots, x_n]$ is a finite module 
over the commutative subalgebra $Z:= \Bbbk[x_1^2, \dots, x_n^2]$ 
and is generated as a $Z$-module by elements of degree $\leq n$,
we have $d=2$ and $m=n$ in Proposition \ref{xxpro3.11}. 
By Proposition \ref{xxpro3.11}(2),
$\deg T/J_{G,i}\leq 2|G|+n-1$. Let $D=D_i=2|G|+n+1$
(which is independent of $i$). 
Then by Proposition \ref{xxpro3.11}(1), 
$\beta(R)\leq \tau_{G}(T)\leq 2|G|+n=D-1$. Now it is routine 
to check that $D(=D_i)$ is at least
equal to each term in the max-expressions in Corollary \ref{xxcor5.5}
and Proposition \ref{xxpro5.8}(d).
By Proposition \ref{xxpro5.8} with $(B,A,C)=(T,S,R)$, we obtain that
$$t^R_i(\Bbbk)\leq i(2|G|+n+1)-2$$
as $t^T_2(\Bbbk)=2$.
\end{proof}

To conclude this section we prove a version of \cite[Theorem 2]{De}. 

\begin{theorem}
\label{xxthm5.11}
Let $(T,H)$ be as in Hypothesis \ref{xxhyp0.3} and suppose that $T$ is AS regular.
Let $R=T^H$. Suppose further that
\begin{enumerate}
\item[(a)]
$T$ is generated in degree 1.
\item[(b)]
$S$ is a noetherian AS regular algebra such that the minimal generating 
vector spaces of $S$ and $R$ have the same dimension
and there exists a graded algebra surjection $S \to R$. 
\end{enumerate}
Then we can conclude:
\begin{enumerate}
\item[(1)]
We have 
$$\begin{aligned}
\beta_2(R):=t^R_2(\Bbbk) &\leq \tau_H(T)+\tau^{\op}_H(T)-\CMreg(T)\\
&\leq 2-2\CMreg(S)+\CMreg(T).
\end{aligned}$$
\item[(2)]
Suppose that $\Tor^S_1(\Bbbk, R)\otimes_R \Bbbk \cong 
\Tor^S_1(\Bbbk, R)$. 
Then 
$$\begin{aligned}
t^S_1(_S R)&\leq \tau_H(T)+\tau^{\op}_H(T)-\CMreg(T)\\
&\leq 2-2\CMreg(S)+\CMreg(T).
\end{aligned}$$
\item[(3)]
Suppose the hypothesis of part {\rm{(2)}}.
Let $K$ be the kernel of the algebra map $S\to R$. Then, as a left
ideal of $A$, $K$ is generated in degree at most
$$\tau_H(T)+\tau^{\op}_H(T)-\CMreg(T)\leq 2-2\CMreg(S)+\CMreg(T).$$
\end{enumerate}
\end{theorem}

Recall that if $T$ is AS regular, then, by Example \ref{xxex2.10}(3),
$\CMreg(T)\leq 0$. 
The condition that $\Tor^S_1(\Bbbk, R)\otimes_R \Bbbk \cong 
\Tor^S_1(\Bbbk, R)$ is automatic when $R$ and $S$ are 
commutative. It holds even for some noncommutative cases;
see, for example, Lemma \ref{xxlem5.12}.

\begin{proof}[Proof of Theorem \ref{xxthm5.11}]
(1) Let $\{f_1,\cdots, f_r\}$ be a set of homogeneous
elements in $R$ that generates $R$ minimally; these
are also elements in $T$. Let $d_i=\deg (f_i)$ 
for all $1 \leq i \leq r$. 

We consider the left $T$-module $U$ defined by 
\begin{equation}
\label{E5.11.1}\tag{E5.11.1}
U:=\left\{(w_1,\cdots,w_r) \in T(-d_1)\oplus \cdots \oplus T(-d_r)
\;\;\; \middle| \; \;\; \sum_{i=1}^w w_i f_i=0\right\}.
\end{equation}
Then $U$ fits into the short exact sequence of graded left $T$-modules,
\begin{equation}
\label{E5.11.2}\tag{E5.11.2}
0 \to U\to \bigoplus_{j=1}^r T(-d_j)\to TR_{\geq 1}(=: J_H(T))\to 0.
\end{equation}
By Lemma \ref{xxlem3.2}(3), for all $ 1 \leq j \leq r$, 
$d_j\leq \tau_{H}(T)$. Applying Lemma \ref{xxlem2.14}(1) to the 
exact sequence
$$0\to J_H(T)\to T \to T/J_H(T)\to 0,$$
and using the fact that $\CMreg(T)\leq 0$, we obtain that
$$\begin{aligned}
\CMreg(J_H(T)) &\leq \max\{\CMreg(T),\CMreg(T/J_H(T))+1\}\\
&=\CMreg(T/J_H(T))+1 =\tau_H(T).
\end{aligned}
$$
Applying Lemma \ref{xxlem2.14}(1) to \eqref{E5.11.2} 
and using the fact that each $d_j\leq \tau_{H}(T)$
(or equivalently, $\CMreg(T(-d_j))\leq \tau_{H}(T)$), we have 
$\CMreg(U)\leq \tau_H(T)+1$. By Theorem \ref{xxthm2.13}(3),
we have
$$\Extreg(J_H(T))=\CMreg(J_H(T))-\CMreg(T)\leq \tau_H(T)-\CMreg(T)$$ 
and 
$$\Extreg(U)=\CMreg(U)-\CMreg(T)\leq \tau_H(T)+1-\CMreg(T).$$ 
Since Ext-regularity is equal to Tor-regularity
[Definition \ref{xxdef2.11}], $U$ is generated in degrees 
$\leq \tau_H(T)+1-\CMreg(T)$ as a left $T$-module,
or
\begin{equation}
\label{E5.11.3}\tag{E5.11.3}
U=\sum_{\lambda\leq  \tau_H(T)+1-\CMreg(T)} T U_{\lambda}.
\end{equation}

There is an induced $H$-action on the 
left $T$-module $\bigoplus_{j=1}^r T(-d_j)$ that makes
\eqref{E5.11.2} a short exact sequence of left $H$-modules.
Consider the left $R$-module $M$ defined by
$$M:= \left\{(w_1,\cdots,w_r) \in R(-d_1)\oplus \cdots \oplus R(-d_r)
\;\; \middle| \; \; \sum_{i=1}^w w_i f_i=0\right\}.$$
Since the $f_i$ are $H$-invariants, \eqref{E5.11.2} is an exact 
sequence of $H$-equivariant $T$-modules, so we can apply 
$(-)^H$. Since $H$ is semisimple, $(-)^H$ 
is an exact functor. Then the following exact sequence follows 
from \eqref{E5.11.2}.
\begin{equation}
\label{E5.11.4}\tag{E5.11.4}
0\to M\to \bigoplus_{j=1}^r R(-d_j)\to R_{\geq 1}\to 0.
\end{equation}
Thus $M$ fits into the short exact sequence,
$$0\to M\to \bigoplus_{j=1}^r R(-d_j)\to R \to \Bbbk\to 0.$$
Let $\fm=R_{\geq 1}$. We can identify $M/\fm M$ with 
$\Tor^R_2(\Bbbk, \Bbbk)$. 

Now we consider $M$ as an $R$-submodule of $U$. Let $J^{\op}$ be 
$R_{\geq 1} T$. Then $J^{\op}=\sum_{j} f_j T=\fm T$. 
Since $H$ is semisimple, applying
$(-)^H$ to the exact sequence
$$0\to J^{\op} U \to U \to U/J^{\op} U\to 0$$
we obtain an exact sequence
\begin{equation}
\label{E5.11.5}\tag{E5.11.5}
0\to (J^{\op} U)^H \to U^H \to (U/J^{\op} U)^H\to 0.
\end{equation}
We have already seen that $U^H=M$. We claim that 
$(J^{\op} U)^H=\fm M$. Let $\phi\in (J^{\op}U)^H$,
which can be written as 
$$\phi=\sum_{j} f_j u_j$$
for some $u_j\in U$. Let $e$ be the integral of $H$.
Then 
$$\phi=e\cdot \phi=
\sum_j f_j (e\cdot u_j)\in \fm M.$$
Now \eqref{E5.11.5} shows that $(U/J^{\op} U)^H\cong M/\fm M$. 

Next we claim that $\deg (U/J^{\op}U)\leq \tau_H(T)+\tau^{\op}_H(T)
-\CMreg(T)$. Each $h\in U$ of degree strictly 
larger than $\tau_H(T)+\tau^{\op}_H(T)-\CMreg(T)$ can be written as
$$h=\sum_i p_i q_i,$$
where $q_i\in U_{\lambda}$ with $\lambda\leq \tau_H(T)+1-\CMreg(T)$
and $p_i\in T_{d}$ with $d\geq \tau^{\op}_H(T)$. By the definition
of $\tau^{\op}_H(T)$, we have $p_i\in J^{\op}$. Hence
$h\in J^{\op}U$. Therefore we proved the claim. Since 
$M/\fm M$ is a subspace of $U/J^{\op}U$, we obtain
that $\deg (M/fm M)\leq \tau_H(T)+\tau^{\op}_H(T)-\CMreg(T)$
or that $\deg(\Tor^R_2(\Bbbk,\Bbbk))\leq \tau_H(T)+\tau^{\op}_H(T)-\CMreg(T)$.
The second inequality follows from Lemma \ref{xxlem2.16}.

(2) The assertion follows from Lemma \ref{xxlem5.1}(3) 
(or \eqref{E5.1.1}) and part (1).

(3)  By the exact sequence $0\to K\to S \to R\to 0$, 
$K$ is generated by elements corresponding to 
$\Tor^S_1(\Bbbk,R)$. Then the assertion follows 
from part (2).
\end{proof}

\begin{lemma}
\label{xxlem5.12}
Suppose that $S$ and $R$ are connected graded algebras and 
$f: S \to R$ is a graded algebra surjection with kernel $K$. 
If $K$ is generated by normal elements in $S$,
then $\Tor^S_1(\Bbbk, R)\otimes_R \Bbbk \cong \Tor^S_1(\Bbbk, R)$.
As a consequence, \eqref{E5.1.1} holds.
\end{lemma}

\begin{proof} Let $K$ be the kernel of the surjective map.
Then 
$\Tor^S_1(\Bbbk, R)\otimes_R \Bbbk \cong \Tor^S_1(\Bbbk, R)$
is equivalent to $S_{\geq 1} K\supseteq K S_{\geq 1}$. 

Write $K=\sum_i w_i S=\sum_i S w_i$ for a set of normal elements
$\{w_i\}\subseteq S$. Then, for each $i$, $w_i S_{\geq 1}
=S_{\geq 1} w_i$. Therefore 
$$K S_{\geq 1}=\sum_i S w_i S_{\geq 1}=\sum_i S S_{\geq 1} w_i
=\sum_i S_{\geq 1} S w_w=S_{\geq 1} K$$
as desired.
\end{proof}

\section{Further questions}
\label{xxsec6}

We conclude by posing some further questions on degree bounds. 
Suppose that $A$ is a noetherian connected graded algebra and 
$H$ is a semisimple Hopf algebra acting homogeneously on $A$. 
Recall the definition of the $\tau$-saturation degree 
$\tau_H(A)$ [Definition~\ref{xxdef1.1}(2)].

\begin{question}
\label{xxque6.1}
Under what hypothesis is $\tau_H(A)=\tau^{\op}_H(A)$?
\end{question}

In Corollary~\ref{xxcor3.3}, we showed that an upper bound 
on $\tau_H(A)$ provides an upper bound for $\beta(A^H)$, the 
maximum degree of a minimal generating set of $A^H$. In 
Example~\ref{xxex3.4}, Theorem~\ref{xxthm3.5}, and 
Proposition~\ref{xxpro3.11}, we were able to compute bounds 
on $\tau_H(A)$. 

\begin{question}
\label{xxque6.2} 
For which $A$ and $H$ can we bound $\tau_H(A)$?
\end{question}

If one is able to bound $\beta(A^H)$, there also remains 
the question of the sharpness of the bound. Noether's bound 
is sharp in the non-modular case. However, in \cite{DH}, 
Domokos and Heged\"{u}s show that if $T$ is a commutative 
polynomial ring over a field of characteristic zero and $G$ 
is a finite group that is not cyclic, then there is a strict 
inequality $\beta(T^G) < |G|$.

\begin{question}
\label{xxque6.3}
Is there a version of Domokos and Heged\"{u}s's result for 
skew polynomial rings (or other AS regular algebras) under 
group (or Hopf) actions?
\end{question}

We saw in Example~\ref{xxex3.6} that for a group $G$, there 
is no universal bound on $\beta(T^G)$ over all AS regular 
algebras that depends only upon the order of the group and 
the degree of the representation, even for a group of order 2. 
This is in contrast to the commutative case. However, we pose 
the following question.

\begin{question}
\label{xxque6.4}
Let $T:=\Bbbk[V]$ be a commutative polynomial ring over a 
field of characteristic zero and $V$ a representation of a 
finite group $G$. Define 
$$\beta(G,V) :=\min\{d: \Bbbk[V]^G 
\text{ generated by elements of degree } \leq d\}$$
$$\beta(G) = \max\{ \beta(G,V): V 
\text{ is a finite representation of } G \}.$$
It is a theorem of Weyl that $\beta(G) = \beta(G,V_{\rm reg})$. 

Is there a version of this result, i.e., a particular 
representation of $G$ or $H$, which has the highest degrees 
of minimal generating invariants for particular families 
of AS regular algebras? 
For example, 
one could fix $A = \Bbbk_{-1}[x_1, \dots, x_n]$.
Then if $|G| = n$, the regular representation
of $G$ induces an action on $A$
and \cite[Theorem 2.5]{KKZ6} gives a bound on 
$\beta(A^G)$.
Does this give a bound for all 
actions of $G$ on $(-1)$-skew polynomial rings?
\end{question}

In this paper, we have focused on actions by semisimple Hopf algebras.
The group algebra $\Bbbk G$ is semisimple precisely in the 
non-modular case (i.e., when the characteristic of $\Bbbk$ does 
not divide $|G|$). Hence, as noted in the introduction,
when $T = \Bbbk[x_1, \dots, x_n]$ the bounds on $\beta(T^G)$ 
depend on whether or not $\Bbbk G$ is semisimple.

\begin{question}If $H$ is a non-semisimple Hopf algebra acting
on a connected graded noetherian AS Gorenstein algebra $T$, 
what bounds can be established on $\beta(T^H)$?
\end{question}

We refer the reader to \cite{CWZ} for examples of non-semisimple 
Hopf algebra actions on AS regular algebras.

\subsection*{Acknowledgments} We thank the referee for carefully
reading the manuscript and making several helpful suggestions.
This work was begun while E. Kirkman
was a visitor at the University of Washington Department of Mathematics;  
she thanks the department for its gracious hospitality.
R. Won was partially supported by an AMS--Simons Travel Grant.
J. J. Zhang was partially supported by the US National Science 
Foundation (Nos. DMS-1700825 and DMS-2001015).

\end{document}